\pdfoutput=1 

\documentclass[11pt,reqno]{amsart}

\usepackage{amssymb,amsmath,amsthm,graphicx,xcolor,mathtools,mathrsfs,tabularx,bbm,tikz,url}
\usepackage{mathabx}\changenotsign
\usepackage{dsfont}
\usepackage[shortlabels]{enumitem}
\usepackage{lmodern}
\usepackage[babel]{microtype}
\usepackage[british]{babel}
\usepackage[utf8]{inputenc}
\allowdisplaybreaks
\usepackage{thm-restate}

\usepackage[backref, hypertexnames=false]{hyperref} 
\hypersetup{
	colorlinks,
	linkcolor={red!60!black},
	citecolor={green!60!black},
	urlcolor={blue!60!black}
}

\def\ssign{\textsection\nobreak\hspace{1pt plus 0.3pt}}
\makeatletter
\let\origsection=\section 
\def\mysection{\@mystartsection{section}{1}\z@{.7\linespacing\@plus\linespacing}{.5\linespacing}{\normalfont\scshape\centering\ssign}}
\def\section{\@ifstar{\origsection*}{\mysection}}
\def\appendix{\par\c@section\z@ \c@subsection\z@
	\let\sectionname\appendixname
	\let\section=\origsection
	\def\thesection{\@Alph\c@section}}
\def\@mystartsection#1#2#3#4#5#6{\if@noskipsec \leavevmode \fi
	\par \@tempskipa #4\relax
	\@afterindenttrue
	\ifdim \@tempskipa <\z@ \@tempskipa -\@tempskipa \@afterindentfalse\fi
	\if@nobreak \everypar{}\else
	\addpenalty\@secpenalty\addvspace\@tempskipa\fi
	\@dblarg{\@mysect{#1}{#2}{#3}{#4}{#5}{#6}}}
\def\@mysect#1#2#3#4#5#6[#7]#8{\edef\@toclevel{\ifnum#2=\@m 0\else\number#2\fi}\ifnum #2>\c@secnumdepth \let\@secnumber\@empty
	\else \@xp\let\@xp\@secnumber\csname the#1\endcsname\fi
	\@tempskipa #5\relax
	\ifnum #2>\c@secnumdepth
	\let\@svsec\@empty
	\else
	\refstepcounter{#1}\edef\@secnumpunct{\ifdim\@tempskipa>\z@ \@ifnotempty{#8}{\@nx\enspace}\else
		\@ifempty{#8}{.}{\@nx\enspace}\fi
	}\@ifempty{#8}{\ifnum #2=\tw@ \def\@secnumfont{\bfseries}\fi}{}\protected@edef\@svsec{\ifnum#2<\@m
		\@ifundefined{#1name}{}{\ignorespaces\csname #1name\endcsname\space
		}\fi
		\@seccntformat{#1}}\fi
	\ifdim \@tempskipa>\z@ \begingroup #6\relax
	\@hangfrom{\hskip #3\relax\@svsec}{\interlinepenalty\@M #8\par}\endgroup
	\ifnum#2>\@m \else \@tocwrite{#1}{#8}\fi
	\else
	\def\@svsechd{#6\hskip #3\@svsec
		\@ifnotempty{#8}{\ignorespaces#8\unskip
			\@addpunct.}\ifnum#2>\@m \else \@tocwrite{#1}{#8}\fi
	}\fi
	\global\@nobreaktrue
	\@xsect{#5}}
\makeatother

\def\rmlabel{\upshape({\kern-0.0833em\itshape\roman*\kern+0.0833em})}
\def\nlabel{\upshape(\kern-0.0833em{\itshape\arabic*}\kern0.0833em)}
\def\alabel{\upshape({\kern-0.0833em\itshape\alph*\kern0.0833em})}
\def\Alabel{\upshape({\kern-0.0833em\itshape\Alph*\kern0.1111em})}

\usepackage{pgfplots}
\pgfplotsset{compat=1.15}
\usepackage{mathrsfs}
\usetikzlibrary{arrows}

\usetikzlibrary{arrows.meta,    
	positioning,
	quotes}         

\usepackage{setspace}

\usepackage{fullpage}

\def\rmlabel{\upshape({\itshape \roman*\,})}

\def\alabel{\upshape({\itshape \alph*\,})}
\def\Alabel{\upshape({\itshape \Alph*\,})}
\def\nlabel{\upshape({\itshape \arabic*\,})}

\let\setminus=\smallsetminus
\let\emptyset=\varnothing
\let\vn=\varnothing

\makeatletter
\def\moverlay{\mathpalette\mov@rlay}
\def\mov@rlay#1#2{\leavevmode\vtop{   \baselineskip\z@skip \lineskiplimit-\maxdimen
		\ialign{\hfil$\m@th#1##$\hfil\cr#2\crcr}}}
\newcommand{\charfusion}[3][\mathord]{
	#1{\ifx#1\mathop\vphantom{#2}\fi
		\mathpalette\mov@rlay{#2\cr#3}
	}
	\ifx#1\mathop\expandafter\displaylimits\fi}
\makeatother

\usepackage{hyperref}
\hypersetup{colorlinks, linkcolor={red!50!black}, citecolor={green!50!black}, urlcolor={blue!50!black}}

\usepackage{caption}
\usepackage{subcaption}

\usepackage[mathlines]{lineno}
\usepackage{etoolbox} 

\newcommand*\linenomathpatch[1]{%
	\expandafter\pretocmd\csname #1\endcsname {\linenomath}{}{}%
	\expandafter\pretocmd\csname #1*\endcsname{\linenomath}{}{}%
	\expandafter\apptocmd\csname end#1\endcsname {\endlinenomath}{}{}%
	\expandafter\apptocmd\csname end#1*\endcsname{\endlinenomath}{}{}%
}
\newcommand*\linenomathpatchAMS[1]{%
	\expandafter\pretocmd\csname #1\endcsname {\linenomathAMS}{}{}%
	\expandafter\pretocmd\csname #1*\endcsname{\linenomathAMS}{}{}%
	\expandafter\apptocmd\csname end#1\endcsname {\endlinenomath}{}{}%
	\expandafter\apptocmd\csname end#1*\endcsname{\endlinenomath}{}{}%
}

\expandafter\ifx\linenomath\linenomathWithnumbers
\let\linenomathAMS\linenomathWithnumbers
\patchcmd\linenomathAMS{\advance\postdisplaypenalty\linenopenalty}{}{}{}
\else
\let\linenomathAMS\linenomathNonumbers
\fi

\linenomathpatchAMS{gather}
\linenomathpatchAMS{multline}
\linenomathpatchAMS{align}
\linenomathpatchAMS{alignat}
\linenomathpatchAMS{flalign}
\linenomathpatch{equation}

\usepackage{cleveref}

\theoremstyle{plain}
\newtheorem{theorem}{Theorem}[section]
\crefname{theorem}{Theorem}{Theorems}

\newtheorem{proposition}[theorem]{Proposition}
\crefname{proposition}{Proposition}{Propositions}

\newtheorem{corollary}[theorem]{Corollary}
\crefname{corollary}{Corollary}{Corollaries}

\newtheorem{lemma}[theorem]{Lemma}
\crefname{lemma}{Lemma}{Lemmata}

\crefname{conjecture}{Conjecture}{Conjectures}

\crefname{problem}{Problem}{Problem}

\newtheorem{claim}[theorem]{Claim}
\crefname{claim}{Claim}{Claims}

\newtheorem{observation}[theorem]{Observation}
\crefname{observation}{Observation}{Observations}

\crefname{setup}{Setup}{Setups}

\newtheorem{fact}[theorem]{Fact}
\crefname{fact}{Fact}{Facts}

\crefname{algorithm}{Algorithm}{Algorithms}

\crefname{remark}{Remark}{Remarks}

\crefname{example}{Example}{Examples}

\theoremstyle{definition}
\newtheorem{definition}[theorem]{Definition}
\crefname{definition}{Definition}{Definitions}

\crefname{construction}{Construction}{Constructions}

\crefname{question}{Question}{Questions}

\numberwithin{equation}{section}

\crefformat{section}{\S#2#1#3}
\crefformat{subsection}{\S#2#1#3}
\crefformat{subsubsection}{\S#2#1#3}
\crefrangeformat{section}{\S\S#3#1#4 to~#5#2#6}
\crefmultiformat{section}{\S\S#2#1#3}{ and~#2#1#3}{, #2#1#3}{ and~#2#1#3}

\crefname{appendix}{Appendix}{Appendix}

\crefname{figure}{Figure}{Figures}

\newcommand{\rf}[1]{\cref{#1} (\nameref*{#1})}

\theoremstyle{definition}

\newtheorem*{space-prop}{Space}

\def\rmlabel{\upshape({\itshape \roman*\,})}

\def\alabel{\upshape({\itshape \alph*\,})}
\def\Alabel{\upshape({\itshape \Alph*\,})}
\def\nlabel{\upshape({\itshape \arabic*\,})}

\newenvironment{proofclaim}[1][Proof of the claim]{\begin{proof}[#1]\renewcommand*{\qedsymbol}{\(\blacksquare\)}}{\end{proof}}

\def\COMMENT#1{}

\usepackage{comment}

\let\polishlcross=\l
\def\l{\ifmmode\ell\else\polishlcross\fi}
\newcommand{\vecb}{\mathbf} 
\newcommand{\es}{\emptyset}
\newcommand{\eps}{\varepsilon}
\renewcommand{\rho}{\varrho}
\newcommand{\sm}{\setminus}
\renewcommand{\subset}{\subseteq}

\newcommand{\NATS}{\mathbb{N}}
\newcommand{\NATSZ}{\mathbb{N}_0}
\newcommand{\INTS}{\mathbb{Z}}

\DeclareMathOperator{\Hom}{{Hom}} 

\newcommand{\Exp}{\mathbb{E}}
\newcommand{\ori}[1]{\smash{\overrightarrow{#1}}}
\let\vn\relax
\newcommand{\vn}{\mathbf{1}}

\let\th\relax
\DeclareMathOperator{\th}{\delta}

\newcommand{\cC}{\mathcal{C}}
\newcommand{\cD}{\mathcal{D}}

\newcommand{\cF}{\mathcal{F}}
\newcommand{\cG}{\mathcal{G}}

\newcommand{\cL}{\mathcal{L}}

\newcommand{\cP}{\mathcal{P}}
\newcommand{\cQ}{\mathcal{Q}}

\newcommand{\cT}{\mathcal{T}}

\newcommand{\cV}{\mathcal{V}}
\newcommand{\cW}{\mathcal{W}}

\DeclareMathOperator{\cover}{{cov}}
\DeclareMathOperator{\res}{{res}}
\DeclareMathOperator{\divn}{{\div}}

\let\div\relax

\DeclareMathOperator{\div}{\mathsf{Div}}
\DeclareMathOperator{\con}{\mathsf{Con}}
\DeclareMathOperator{\spa}{\mathsf{Spa}}

\DeclareMathOperator{\ham}{\mathsf{HC}}
\DeclareMathOperator{\frt}{\mathsf{FT}}
\DeclareMathOperator{\frct}{\mathsf{FCT}}

\DeclareMathOperator{\hf}{\mathsf{HF}}

\DeclareMathOperator{\dcon}{{\con}}
\DeclareMathOperator{\dspa}{{\spa}}

\DeclareMathOperator{\hamcon}{\mathsf{HamCon}}

\DeclareMathOperator{\Del}{\mathsf{Del}}
\DeclareMathOperator{\adh}{\mathsf{adh}}
\DeclareMathOperator{\tc}{\mathsf{c}}
\let\P\relax
\DeclareMathOperator{\P}{\mathsf{P}}
\let\div\relax
\DeclareMathOperator{\div}{\mathsf{Div}}
\newcommand{\PG}[3]{{P^{(#3)}}(#1,#2)}

\DeclareMathOperator{\Deg}{\mathsf{MinDeg}}
\newcommand{\DegF}[2]{\Deg_{#1,#2}}

\newcommand{\tightly}{}
\newcommand{\tight}{}

\title{Loose Hamiltonicity}

\author[R.~Lang]{Richard Lang}

\address[R.~Lang]{
	Departament de Matemàtiques,
	Universitat Politècnica de Catalunya,
	Barcelona, Spain and
	Centre de Recerca Matemàtica, Barcelona, Spain
}
\email{richard.lang@upc.edu}

\author[N.~Sanhueza-Matamala]{Nicolás Sanhueza-Matamala}

\address[N.~Sanhueza-Matamala]{ 
	Departamento de Ingeniería Matemática,
	Facultad de Ciencias Físicas y Matemáticas,
	Universidad de Concepción,
	Concepción, Chile
}
\email{nsanhuezam@udec.cl}

\begin{document}
	
\begin{abstract}
	We study the appearance of Hamilton $\ell$-cycles in dense $k$-uniform hypergraphs when $\ell \leq k-2$ and $k-\ell$ does not divide $k$.
	Our main result reduces this problem to the robust existence of a connected $\ell$-cycle tiling in host graph families that are approximately closed under subsampling.
	As an application, we determine the minimum $d$-degree threshold for $d=k-2$ and all $1 \leq \ell \leq k-2$ when $k - \ell$ does not divide $k$.
	We also reduce the case $\ell < d$ entirely to the corresponding (non-connected) $\ell$-cycle tiling problem.
	In addition, our outcomes lead to counting and random robust versions of these results.
	The proofs are based on the recently introduced method of blow-up covers and thus avoid the use of the Regularity Lemma and the Absorption Method.
\end{abstract}



%
\maketitle


\vspace{-1.0cm}

\section{Introduction}\label{sec:introduction}

We study vertex-spanning $\ell$-cycles in $k$-uniform hypergraphs, a notion of Hamiltonicity that dates back to Baranyai's Wreath Conjecture from the 1970s~\cite{Bar75}.
Over the past two decades, the appearance of Hamilton $\ell$-cycles in dense hypergraphs has been investigated under conditions of minimum degree and quasirandomness, 
mostly relying on combinations of the Regularity Lemma and the Absorption Method.
While there has been progress for particular choices of parameters (such as $k\leq 3$), the problem remains open in most cases.

We recently introduced a set of tools to tackle these problems in the ``tight'' setting, when $\ell=k-1$~\cite{LS24a}.
This article focuses on the ``loose'' case, when $\ell \leq k-2$.
We develop a framework for finding Hamilton $\ell$-cycles in dense hypergraphs in the case when $k-\ell$ does not divide~$k$.
It is shown that the existence of Hamilton $\ell$-cycles can be reduced to solving a connected tiling problem (\cref{thm:framework-hamilton-cycle}).
As an application, we determine the threshold for Hamiltonicity in several new instances and recover all known (asymptotic) bounds.
Moreover, we give an abstract reduction of the Hamiltonicity problem to a tiling problem when $\ell < d$.
Besides embedding a Hamilton cycle, our framework also robustly guarantees ``Hamilton connectedness'' (\cref{thm:framework-connectedness-robust}), which allows us to find Hamilton paths beginning and ending in prescribed sets of vertices.
In combination with our previous work with Joos~\cite{JLS23}, this establishes counting and random robust versions of the aforementioned results.
Our proofs are based on the recently introduced method of blow-up covers~\cite{Lan23,LS24a} and therefore avoid the use of the Regularity Lemma and the Absorption Method.

\subsection*{Background}\label{sec:background+applications}

We introduce a bit of terminology to discuss past work on Hamiltonicity in hypergraphs.
Formally, a $k$-uniform hypergraph, or \emph{$k$-graph} for short, $G$ consists of a set of vertices $V(G)$ and a set of edges $E(G)$, where each edge is a $k$-set of vertices.
We often identify $G$ with its edge set, writing $e \in G$ for an edge $e$.
The \emph{order} $v(G)$ of $G$ is its number of vertices.
For $0 \leq \ell < k$, an \emph{$\ell$-cycle}~$C$ in a $k$-uniform hypergraph $G$ has its vertices cyclically ordered such that each edge consists of $k$ consecutive vertices and consecutive edges intersect in exactly $\ell$ vertices.
Note that the order of $C$ is divisible by $k-\ell$.
Moreover, $C$ is \emph{Hamilton} if it covers all vertices of $G$.
For $1 \leq d < k$, the \emph{minimum $d$-degree $\delta_d(G)$} of $G$ is the maximum $m$ such that every set of $d$ vertices is in at least $m$ edges.
The \emph{minimum $d$-degree threshold for Hamilton $\ell$-cycles} $\delta_d^{\ham}(k,\ell)$ is the infimum of $\delta \in [0,1]$ such that for every $\eps > 0$, there is $n_0$ such that every $k$-graph $G$ on $n \geq n_0$ divisible by $k-\ell$ vertices with $\delta_d(G)\geq (\delta + \eps ) \binom{n-d}{k-d}$ contains a Hamilton $\ell$-cycle.
For instance, Dirac's theorem~\cite{Dir52} implies that $\delta_1(2, 1) = 1/2$.

Hamiltonicity for $\ell$-cycles was initially studied in the ``tight'' setting, when $\ell = k-1$.
For more background on this case, we refer the reader to the summary in the recent work of Lang, Schacht and Volec~\cite{LSV24}.
In the following, we focus on the ``loose'' setting, when $1 \leq \ell \leq k-2$.
Let
\begin{align}
	\lambda(k,\ell) = \frac{1}{\lceil \tfrac{k}{k-\ell} \rceil (k-\ell)}\,. \label{itm:lamda}
\end{align}
Kühn, Mycroft and Osthus~\cite{KMO10} showed that $\delta_{d}^{\ham}(k,\ell) = \lambda(k,\ell)$ for $d=k-1$ whenever $k-\ell \not \mid k$.
Moreover, they also gave a construction, which together with the work of Rödl, Ruciński and Szemerédi~\cite{RRS08a} shows that  $\delta_{k-1}^{\ham}(k,\ell) = 1/2$ when $k-\ell \mid k$.
Prior to our work, all known bounds for $\ell$-cycle thresholds for $d \leq k-2$ assume that $\ell \leq k/2$.
Note that in the case where $\ell < k/2$, this implies that $k-\ell \not \mid k$.
In particular, the threshold for $d = k-2$ and $1 \leq \ell < k/2$ was determined by Bastos, Mota, Schacht, Schnitzer and Schulenburg~\cite{BMSSS17}; and for $\ell = k/2$ this follows from results of Garbe and Mycroft~\cite{GM18} and Hàn, Han and Zhao~\cite{HHZ20}.
Moreover, Gan, Han, Sun and Wang~\cite{GHS+21} and Han, Sun and Wang~\cite{HSW25} obtained results in the case $d = k-3$, for  combinations of $(k, \ell)$ satisfying $1 \leq \ell < k/2$.
Additionally, several exact bounds have been obtained~\cite{BMSSS18, GM18, HZ15b}.

Hamiltonicity in hypergraphs has been studied beyond the minimum degree setting.
Lenz, Mubayi and Mycroft~\cite{LMM16} as well as Hàn, Han and Morris~\cite{HHM20} investigated Hamilton $1$-cycles in the quasirandom setting.
The problem of finding Hamilton $\ell$-cycles in the binomial random $k$-graph was resolved by Dudek and Frieze~\cite{DF13} as well as Narayanan and Schacht~\cite{NS20}.
More recently, Mycroft and Zárate-Guerén~\cite{MZ25} as well as Letzter and Ranganathan~\cite{LR25} determined the positive codegree thresholds for Hamilton $\ell$-cycles.
Beyond this, loose Hamiltonicity has also been studied in the settings of random resilience~\cite{AKL+23,PT22}, random robustness~\cite{JLS23,KMP23}, transversal structures~\cite{GHM+23} and rainbow structures~\cite{KMP25}.

\subsection*{Outcomes}\label{sec:results}

In the following, we discuss the applications of our framework, which itself is presented in \cref{sec:frameworks}.
Recall the definition of $\lambda(k,\ell)$ in~\eqref{itm:lamda}.
As a warm-up, we recover the following result of Kühn, Mycroft and Osthus~\cite{KMO10} for codegree conditions (when $d = k-1$).

\begin{theorem}\label{thm:d=k-1}
	For $1 \leq \ell \leq k-2$ with $k-\ell \not \mid k$, we have $\delta_{k-1}^{\ham}(k,\ell) = \lambda(k,\ell).$
\end{theorem}

Turning to the case $d = k-2$,
Bastos, Mota, Schacht, Schnitzer and Schulenburg~\cite{BMSSS17} proved that for all $k \geq 3$ and $1 \leq \ell < k/2$, we have $\delta_{k-2}^{\ham}(k,\ell) = \frac{4(k-\ell)-1}{4(k-\ell)^2} = 1- (1-\lambda(k,\ell))^2$, extending a previous result of Buß, Hàn and Schacht~\cite{BHS13} for the case $k=3$.
(As mentioned above, $1 \leq \ell < k/2$ implies $k-\ell \not \mid k$.)
Our first application generalises this result to the whole range of $1 \leq \ell \leq k-2$, and is the first result for $d=k-2$ and $\ell > k/2$.

\begin{theorem}\label{thm:d=k-2}
	For $k \geq 3$ and $1 \leq \ell \leq k-3$ with $k - \ell \not \mid k$, we have
	\begin{align*}
		\delta_{k-2}^{\ham}(k,\ell) =   1- (1-\lambda(k,\ell))^2 \,.
	\end{align*}
	Moreover, if $\ell = k-2$ and $k$ is odd, then
	\begin{align*}
		\delta_{k-2}^{\ham}(k,\ell) = \max \left\{\frac{1}{4},\,1- (1-\lambda(k,\ell))^2 \right\}= \begin{cases}
			\frac{7}{16} & \text{ if $k=3$}\,, \\
			\frac{11}{36} & \text{ if $k = 5$}\,, \\
			\frac{1}{4}  & \text{ if $k \geq 7$} \,.
		\end{cases}
	\end{align*}
\end{theorem}
{Our contribution is the proof of the upper bounds.
The lower bound involving $\lambda(k,\ell)$ is  due to Kühn, Mycroft and Osthus~\cite{KMO10} and comes from `space barrier' constructions.
The bound of $1/4$ is due to Han and Zhao~\cite[Theorem 1.5]{HZ16} and exploits the connectedness of $\ell$-cycles.}

Our second application is of a 	more abstract nature and reduces the problem of Hamiltonicity to a problem of perfect fractional cycle tilings.
To motivate the definition of fractional tilings, let us first introduce the integral one.
Given a family of $k$-graphs $\cF$, an \emph{$\cF$-tiling} in a $k$-graph $G$ is a collection of pairwise disjoint copies of (possibly distinct) $k$-graphs in $\cF$.
Moreover, the tiling is \emph{perfect} if all vertices of $G$ are covered.
Moving to the fractional setting, a \emph{homomorphism} from a $k$-graph $F$ to $G$ is a function $\phi\colon V(F) \to  V(G)$ that maps edges to edges.
Denote by $\Hom(\cF,G)$ the set of all homomorphisms from $F$ to $G$ with $F \in \cF$.
A \emph{perfect fractional $\cF$-tiling} is a function $\omega\colon \Hom(\cF,G) \to [0,1]$ such that $\sum_{\phi \in \Hom(\cF,G)} \omega(\phi) |\phi^{-1}(v)| = 1$ for all $v \in V(G)$.
Note that a perfect $\mathcal{F}$-tiling corresponds to an integral perfect fractional $\mathcal{F}$-tiling.
When $\cF$ is the family of $\ell$-cycles, then we also speak of a \emph{fractional $\ell$-cycle tiling}.
So a Hamilton $\ell$-cycle is a particular instance of a fractional $\ell$-cycle tiling.
Finally, we define the threshold $\delta_d^{\frct}(k, \ell)$ as the infimum of $\delta \in [0,1]$ such that for every $\eps > 0$, there is $n_0$ such that every $k$-graph $G$ on $n \geq n_0$  vertices with $\delta_d(G)\geq (\delta + \eps ) \binom{n-d}{k-d}$ contains a perfect fractional $\ell$-cycle tiling.
By the observation above, it is evident that $\delta_{d}^{\frct}(k, \ell) \leq \delta_{d}^{\ham}(k,\ell)$.
Our second outcome describes conditions for which this is an equality.

\begin{theorem}\label{thm:cycle-threshold=tiling-treshold}
	For $1 \leq \ell < d \leq k-1$ with $k-\ell \not \mid k$, we have $\delta_{d}^{\ham}(k,\ell) = \delta_{d}^{\frct}(k, \ell).$
\end{theorem}

We note that Han, Sun and Wang~\cite[Corollary 1.6]{HSW25} recently proved a similar result under the additional constraint that $\ell < d \leq 0.18k+0.82 \ell$,
 and the involved tilings use only copies of a specific (small) $\ell$-cycle.
It follows from the work of the first author~\cite[Theorem 5.8]{Lan23} that $\delta_{k-1}^{\frct}(k, \ell) = \lambda(k,\ell)$ and $\delta_{k-2}^{\frct}(k, \ell) = 1- (1-\lambda(k,\ell))^2$.
(For the sake of completeness, the details of this statement can be found in \cref{sec:applications-proofs}.)
So we obtain \cref{thm:d=k-1} and the first part of \cref{thm:d=k-2} as an immediate corollary of \Cref{thm:cycle-threshold=tiling-treshold}.
On the other hand, the threshold for perfect cycle tilings is too crude to capture a bound such as $\delta_{5}^{\ham}(7,5) = 1/4$ as stated in \cref{thm:d=k-2}.
To see why, we need a better understanding of the connectivity structure, which we discuss in the next section.

As a consequence of a more general setup and previous work with Joos~\cite{JLS23} (see~\cref{sec:connectedness-results}), we can strengthen our outcomes to the random robust setting, where edges of a dense host graph are randomly deleted.
For instance, we obtain the following random robust version of \cref{thm:d=k-2}:

\begin{theorem}\label{thm:d=k-2-random-robust}
	For $k\geq 3$, $2 \leq \ell \leq k-2$ with $k - \ell \not \mid k$ and $\eps > 0$, let $G$ be a $k$-graph on $n$  divisible by $k-\ell$ vertices with $\delta_{k-2}(G) \geq ( \th_{k-2}^{\ham}(k, \ell) + \eps ) \binom{n-d}{k-d}$.
	If $G_p \subseteq G$ is obtained by independently keeping every edge with probability $p = \omega(n^{-k+\ell})$, then with probability tending to $1$ as $n$ goes to infinity, $G_p$ contains a Hamilton $\ell$-cycle.
\end{theorem}

We remark that the value of $p$ is essentially optimal~\cite{JLS23}.
In a similar way, we also obtain counting versions of our outcomes, which gives, for instance, the following variant of \cref{thm:d=k-2}:

\begin{theorem}\label{thm:d=k-2-counting}
	For $k\geq 3$, $1 \leq \ell \leq k-2$ and $\eps>0$, there exists $C > 0$ such that the following holds.
	Let $G$ be a $k$-graph on $n$  divisible by $k-\ell$ vertices with $\delta_{k-2}(G) \geq ( \th_{k-2}^{\ham}(k, \ell) + \eps ) \binom{n-d}{k-d}$.
	Then~$G$ contains at least $\exp(n \log n - Cn)$ Hamilton $\ell$-cycles.
\end{theorem}

Finally, we present a separate result, which addresses the case when $k - \ell$ divides $k$.
In the codegree case ($d = k-1$), we have $\delta_{k-1}^{\ham}(k, \ell) = 1/2$, where the lower bound follows from the fact that, if $n$ is divisible by $k$ and $k - \ell$ divides $k$, then the existence of a Hamilton $\ell$-cycle in an $n$-vertex $k$-graph forces the existence of a perfect matching.
The next result allows us to deduce similar bounds for degree conditions beyond codegrees, by a reduction to hypergraphs of smaller uniformity.

\begin{theorem} \label{thm:generalsquash}
	For $1 \leq d, \ell \leq k-1$ and $q \geq 2$, we have $\th_{qd}^{\ham}(qk,q\ell) \leq \th_{d}^{\ham}(k,\ell)$.
\end{theorem}

Using this result together with the known thresholds, one easily obtains the following.

\begin{corollary}\label{cor:squash}
	Let $1 \leq \ell \leq k-1$ such that $k - \ell$ divides $k$ and $d,q\geq2$ with $(k-1)q \leq d < qk$.
	Then $\th_{d}^{\ham}(qk,q\ell) = 1/2$.
\end{corollary}

{Let us mention a few consequences of \cref{cor:squash}.
Applied with $(k, \ell, q) = (2, 1, q)$ for a $q \geq 2$, we obtain that $\th_{d}^{\ham}(2q,q) = 1/2$ holds for all $q \leq d < 2q$.
This recovers the known results for those cases~\cite{GM18,HHZ20}.
Moreover, for $(k, \ell, q) = (k, k-1, q)$ with $k \geq 3$ and $q \geq 2$, this gives, to the best of our knowledge, new results.
For instance, for $k = 3$ and $q = 2$, we obtain $\th_{5}^{\ham}(6,4) = \th_{4}^{\ham}(6,4) = 1/2$.}

\subsection*{Organisation of the paper}

The remainder of the paper is organised as follows.
In the next section, we introduce a general setup to find Hamilton cycles and reduce the proofs of our applications to a simpler setting.
We finish the proofs of the applications in \cref{sec:applications-proofs}.
In~\cref{sec:proof-main-result}, we derive our main result assuming certain allocation results and technical lemmata.
The allocation proofs can be found in \cref{sec:allocation-cycle}.
The proofs of \cref{thm:generalsquash,cor:squash} are given {in~\cref{sec:squashing}}.
We finish with a few concluding remarks in \cref{sec:conclusion}.
Finally, \cref{sec:auxilliary-results-proofs} contains the proofs of two auxiliary lemmata, and \cref{sec:connectedness} contains the details of a result on Hamilton connectedness.

\section{A framework for Hamiltonicity}\label{sec:frameworks}

In this section, we present a series of abstract results, which reduce the problem of finding Hamilton cycles and related questions to a simpler setting.

\subsection{Necessary and sufficient conditions}\label{sec:necessary-conditions}

Let us begin with a discussion of the necessary conditions for `loose' Hamiltonicity.
We can interpret a Hamilton cycle as a perfect cycle tiling.
In this sense, a perfect (fractional) tiling is a precondition for the existence of a Hamilton cycle.
Another necessary condition concerns connectivity.
The \emph{$\ell$-line graph} of $G$ is the $2$-graph on vertex set $E(G)$ with an edge $ef$ whenever $|e \cap f| \geq \ell$.
A subgraph of $G$ is \emph{\tightly $\ell$-connected} if it has no isolated vertices and its edges induce a connected subgraph in the $\ell$-line graph of~$G$.
We refer to edge-maximal \tightly connected subgraphs as \emph{$\ell$-components}.
Note that an $\ell$-cycle is itself $\ell$-connected.
So a Hamilton $\ell$-cycle $C$ in a $k$-graph $G$ can be viewed as a perfect cycle tiling in an $\ell$-connected vertex-spanning subgraph $F \subset G$.
This motivates the following definition, where we in addition require that one can pick the subgraphs $F$ in a consistent way (after small changes to $G$).

\begin{definition}[Hamilton framework]\label{def:hamilton-framework}
	A family $\cG$ of $s$-vertex $k$-graphs has a \emph{Hamilton $\ell$-framework}~$F$ if for every $G \in \cG$ there is an {$s$-vertex} subgraph $F(G) \subset G$ such that
	\begin{enumerate}[(F1)]
		\item \label{itm:hf-connected} $F(G)$ is an \tight $\ell$-component, \hfill(connectivity)
		\item \label{itm:hf-matching} $F(G)$ has a perfect fractional $\ell$-cycle tiling, \hfill(space)
		\item \label{itm:hf-intersecting} $F(H-x) \cup F(H-y)$ is $\ell$-connected for any $(s+1)$-vertex $k$-graph $H$ and $x,y \in V(H)$ such that $H-x, H-y \in \cG$.\hfill(consistency)
	\end{enumerate}
\end{definition} 

The properties of Hamilton $\ell$-frameworks are themselves too fragile to imply Hamiltonicity.
Indeed, it is easy to construct families $\cG$ that admit	 a Hamilton $\ell$-framework, but whose members do not contain Hamilton $\ell$-cycles.
(Consider for instance the $2$-graph obtained by appending an edge to an odd cycle.)
To avoid this, we shall require that the property of admitting Hamilton $\ell$-frameworks is closed under sub-sampling, which adds to its robustness.
This is formalised using the notion of `property graphs', which was introduced in the context of perfect tilings~\cite{Lan23}.

\begin{definition}[Property graph]\label{def:property-graph}
	For a $k$-graph $G$ and a family of $k$-graphs $\P$, the \emph{$s$-uniform property graph}, denoted by $\PG{G}{\P}{s}$, is the $s$-graph on vertex set $V(G)$ with an edge $S \subset V(G)$ whenever the induced subgraph $G[S]$ \emph{satisfies}~$\P$, that is $G[S] \in \P$.
\end{definition}

Thus, an $s$-uniform property graph of $G$ tracks small vertex sets that inherit the property $\P$.
We remark that for all practical purposes, we can assume that $s$ is much larger than $k$, and the order of~$G$ is much larger than $s$.
Given this, we express the robustness of a property $\P$ in a $k$-graph $G$ by saying that the property $s$-graph for $\P$ has sufficiently large minimum degree.

\begin{definition}[Robustness]\label{def:robustness}
	For a family of $k$-graphs $\P$, $r=2k$ and $\delta=1-1/s^2$, a $k$-graph~$G$ \emph{satisfies $s$-robustly} $\P$ if the minimum $r$-degree of the property $s$-graph $\PG{G}{\P}{s}$ is at least~$\delta   \tbinom{n-r}{s-r}$.
\end{definition}

We are now ready to formulate our main result, whose proof can be found in \cref{sec:proof-main-result}.

\begin{theorem}[Main result]\label{thm:framework-hamilton-cycle}
	For  all $1 \leq \ell < k \leq s$ with $k-\ell \not \mid k$, there is $n_0$ such that the following holds.
	Let $\P$ be a family of $s$-vertex $k$-graphs that admits a Hamilton $\ell$-framework.
	Then every $k$-graph $G$ on $n\geq n_0$ divisible by $k-\ell$ vertices that $s$-robustly satisfies $\P$ has a Hamilton $\ell$-cycle.
\end{theorem}

The purpose of \cref{thm:framework-hamilton-cycle} is to reduce the existence of a Hamilton cycle to verifying the (much simpler) necessary conditions of Hamilton frameworks for (hyper)graph families that are robust under sub-sampling.
We illustrate this in the following subsection.

\subsection{Minimum degree thresholds}

Our framework applies to host graphs that satisfy graph properties, which are hereditary in the sense of being approximately closed under taking typical induced subgraphs of constant order.
This is formalised as follows.

\begin{definition}\label{def:degree-family}
	For $1 \leq d \leq k$ and $\delta \geq 0$, the family $\DegF{d}{\delta}$ contains, for all $k \leq n$, every $n$-vertex $k$-graph~$G$ with minimum $d$-degree $\delta_d(G) \geq \delta    \binom{n-d}{k-d}$.
\end{definition}

Some of our constant hierarchies are expressed in standard $\gg$-notation.
To be precise, we write $y \gg x$ to mean
that for any $y \in (0, 1]$ there exists an $x_0 \in (0,1)$
such that for all $x_0 \geq x$ the subsequent statements
hold.  Hierarchies with more constants are defined in a
similar way and are to be read from left to right following the order in which the constants are chosen.
Moreover, we tend to ignore rounding errors if the context allows~it.

\begin{lemma}[Inheritance Lemma]\label{lem:inheritance-minimum-degree}
	For $1/k,\,1/r,\,\eps \gg 1/s \gg 1/n$ and $\delta \geq 0$, let $G$ be an $n$-vertex $k$-graph with $\delta_d(G) \geq (\delta + \eps) \tbinom{n-d}{k-d}$.
	Then the property $s$-graph $P =\PG{G}{\DegF{d}{\delta+\eps/2}}{s}$ satisfies $\delta_{r}(P) \geq  (1-e^{-\sqrt{s}}    )  \tbinom{n-r}{s-r}$.\qed
\end{lemma}

The proof of \cref{lem:inheritance-minimum-degree} follows from standard probabilistic concentration bounds~\cite[Lemma 4.9]{Lan23}.
We can thus restate our main result in the special case of minimum degree conditions.
For $1 \leq \ell \leq k-1$, let $\delta_d^{\hf}(k,\ell)$ be the infimum of $\delta \in [0,1]$ such that for every $\eps > 0$, there is $n_0$ such that the family of $k$-graphs $G$ on $n \geq n_0$ divisible by $k-\ell$ vertices with $\delta_d(G) \geq (\delta + \eps ) \binom{n-d}{k-d}$ admits a Hamilton $\ell$-framework.

\begin{corollary}\label{cor:framework-simple}
	For every $1 \leq d \leq k-1$ and $1 \leq \ell \leq k-2$ with $k-\ell \not \mid k$, there is $n_0$ such that the following holds.
	Every $k$-graph $G$ on $n\geq n_0$ divisible by $k-\ell$ vertices with $\delta_d(G) \geq (\delta_d^{\hf}(k,\ell) + \eps  ) \binom{n-d}{k-d}$ has a Hamilton $\ell$-cycle.
\end{corollary}

\begin{proof}
	Set $\delta = \delta_d^{\hf}(k,\ell)$ and set $r=2k$.
	Given $\eps > 0$, choose $s$ with $1/k,\,\eps \gg 1/s \gg 1/n$.
	Let $G$ be a $k$-graph on $n$ vertices with $\delta_d(G) \geq (\delta + \eps) \binom{n-d}{k-d}$.
	By \cref{lem:inheritance-minimum-degree} the property $s$-graph $P =\PG{G}{\DegF{d}{\delta+\eps/2}}{s}$ satisfies $\delta_{r}(P) \geq  (1-1/s^2)  \tbinom{n-r}{s-r}$.
	Let $\P$ be the family of $s$-vertex $k$-graphs $R$ in $\DegF{d}{\delta+\eps/2}$.
	Then $\P$ admits a Hamilton framework by the definition of $\delta$.
	Thus we can finish by applying \cref{thm:framework-hamilton-cycle}.
\end{proof}

Given \cref{cor:framework-simple}, the results of \cref{sec:introduction}, namely \cref{thm:d=k-1,thm:d=k-2,thm:cycle-threshold=tiling-treshold}, follow immediately from their following framework counterparts (in which $\delta^{\ham} $ is simply replaced by $\delta^{\hf} $).

\begin{proposition}\label{prp:d=k-1}
	For $1 \leq \ell \leq k-1$ with $k-\ell \not \mid k$, we have $\delta_{k-1}^{\hf}(k,\ell) = \lambda(k,\ell).$
\end{proposition}

\begin{proposition}\label{prp:d=k-2}
	For $k \geq 3$ and $1 \leq \ell \leq k-3$ with $k - \ell \not \mid k$, we have
	\begin{align*}
		\delta_{k-2}^{\hf}(k,\ell) =   1- (1-\lambda(k,\ell))^2 \,.
	\end{align*}
	Moreover, if $\ell = k-2$ and $k$ is odd, then
	\begin{align*}
		\delta_{k-2}^{\hf}(k,\ell) = \max \left\{\frac{1}{4},\,1- (1-\lambda(k,\ell))^2 \right\}= \begin{cases}
			\frac{7}{16} & \text{ if $k=3$}\,, \\
			\frac{11}{36} & \text{ if $k = 5$}\,, \\
			\frac{1}{4}  & \text{ if $k \geq 7$} \,.
		\end{cases}
	\end{align*}
\end{proposition} 

\begin{proposition}\label{prp:cycle-threshold=tiling-treshold}
	For $1 \leq \ell < d \leq k-1$ with $k-\ell \not \mid k$, we have $\delta_{d}^{\hf}(k,\ell) = \delta_{d}^{\frct}(k, \ell).$
\end{proposition}

We derive the three  results above in \cref{sec:applications-proofs}.
As we shall see, in these examples the consistency condition of \cref{def:hamilton-framework} is rather trivial to obtain.
We believe that this is generally the case, meaning that Hamilton frameworks appear along with Hamilton cycles:

\begin{restatable}[]{conjecture}{conthresholds}\label{con:thresholds}
	For $1 \leq \ell,d \leq k-1$ with $k-\ell \not \mid  k$, we have  $\delta_{d}^{\hf}(k,\ell) = \delta_{d}^{\frct}(k, \ell)$.
\end{restatable}

\subsection{Hamilton connectedness}\label{sec:connectedness-results}

Next, we turn to a version of our main result, where we aim to connect any two suitable vertex tuples with a Hamilton path.
Formally, an  \emph{$\ell$-\emph{path}} is a $k$-graph that consists of consecutive edges of an $\ell$-cycle such that the first edge and the last edge do not intersect.
Note that the order of an $\ell$-path is congruent to $k$ modulo $k-\ell$.
An $\ell$-path in a $k$-graph~$G$ is \emph{Hamilton} if it covers all vertices of~$G$.
The first and last $\ell$ vertices of an $\ell$-path $P$ are the \emph{endtuples} of $P$; we say that $P$ is an \emph{$(e, f, \ell)$-path}.
It is convenient to use non-uniform hypergraphs to keep track of well-connected $\ell$-tuples.
A \emph{$k$-bounded hypergraph} (or \emph{$[k]$-graph} for short) $G$ consists of a set of vertices $V(G)$ and a set of edges $E(G)$, where each edge is a (non-empty) set of at most $k$ vertices.
We denote by $G^{(i)} \subset G$ the vertex-spanning $i$-uniform subgraph that contains the edges of uniformity $i$.
In the following, we work with a $k$-uniform and an $\ell$-uniform level, where the $\ell$-uniform level encodes sets that are suitable for connection.
An \emph{orientation $\hat{f} \in V(G)^\ell$} of an $\ell$-set $f \subset V(G)$ is an ordering of its vertices expressed as an $\ell$-tuple.

\begin{restatable}[Hamilton connectedness]{definition}{defhamiltonconnectedness}\label{def:hamilton-connectedness}
	For $1 \leq \ell \leq k-1$, denote by $\hamcon_\ell$ the set of $[k]$-graphs $G$ such that $G^{(\ell)}$ contains at least two vertex-disjoint edges and $G^{(k)}$ contains a \tight Hamilton $(\hat{e},\hat{f},\ell)$-path for all orientations $\hat{e},\hat{f} \in V(G)^\ell$ of vertex-disjoint edges $e,f \in G^{(\ell)}$.
\end{restatable}

We remark that every $s$-vertex $k$-graph $R \in \hamcon_\ell$ satisfies $s \equiv k \bmod k-\ell$.
For a $k$-graph $G$, we denote by $\partial_\ell(G)$ the \emph{shadow} $\ell$-graph on $V(G)$ whose edges are the $\ell$-sets that are contained in an edge of $G$.

We can now state our second main result, which guarantees robust Hamilton-connectedness under the same assumptions as in \cref{thm:framework-hamilton-cycle}.
It is proven in \cref{sec:connectedness}.

\begin{restatable}[Hamilton connectedness]{theorem}{thmframeworkconnectednessrobust}\label{thm:framework-connectedness-robust}
	Let $1/\ell,1/k\geq {1/s_1} \gg 1/s_2 \geq 1/s_3 \gg 1/n$.
	Let $\P$ be a family of $s_1$-vertex $k$-graphs that admits a Hamilton $\ell$-framework.
	Let $G$ be a $k$-graph on $n$ vertices
	that satisfies $s_1$-robustly $\P$.
	Then there is an $n$-vertex $[k]$-graph $G' \subset G \cup \partial_\ell (G)$ that satisfies $s$-robustly $\hamcon_\ell$ for every $s_2 \leq s \leq s_3$ with $s \equiv k \bmod k-\ell$.
\end{restatable}

In combination with our work with Joos~\cite{JLS23},  \cref{thm:framework-connectedness-robust} yields random robust versions of \cref{cor:framework-simple} where we find Hamilton $\ell$-cycles in $k$-graphs after the edges are randomly deleted.
In particular, we obtain random robust versions of all the results in \cref{sec:results}.
We state this result for $\ell \geq 2$.
For $\ell = 1$ the quantification is slightly different (as the probability requires an extra logarithmic factor) and it is already stated and proven in past work~\cite[Theorem 2.12]{JLS23} and~\cite[Theorems 1.8]{KMP25}.
The proof of the next corollary follows using \cref{thm:framework-connectedness-robust} as a black box in exactly the same way as~\cite[Theorem 2.11]{JLS23} is proven, so we omit further details.

\begin{corollary}\label{cor:random-robust}
	Let $1 \leq \ell,d  \leq k-1$, with $\ell \geq 2$ and $k - \ell \not\mid k$ and $\eps > 0$.
	Let $G$ be a $k$-graph on $n$ divisible by $k-\ell$ vertices with $\delta_d(G) \geq ( \th_{d}^{\hf}(k, \ell) + \eps ) \binom{n-d}{k-d}$.
	If $G_p \subseteq G$ is obtained by independently keeping every edge with probability $p = \omega(n^{-k+\ell})$, then with probability tending to $1$ as $n$ goes to infinity, $G_p$ contains a Hamilton $\ell$-cycle.
\end{corollary}

We also obtain counting versions of \cref{cor:framework-simple}, which can be easily derived from the existence of spread distributions due to our work with Joos~\cite{JLS23} (see, for example \cite[Section 1.4.1]{KMP23} for such a deduction).

\begin{corollary}\label{cor:counting}
	For $1 \leq \ell,d  \leq k-1$ with $k - \ell \not\mid k$ and $\eps > 0$, there exists $C > 0$ such that the following holds.
	Let $G$ be a $k$-graph on $n$ divisible by $k-\ell$ vertices with $\delta_d(G) \geq ( \th_{d}^{\hf}(k, \ell) + \eps ) \binom{n-d}{k-d}$.
 	Then $G$ contains at least $\exp(n \log n - Cn)$ Hamilton $\ell$-cycles.
\end{corollary}

This recovers and extends the work of Glock, Gould, Joos, Kühn and Osthus~\cite{GGJKO2020}, Montgomery and Pavez-Signé~\cite{MP2023} and Kelly, Müyesser and Pokrovskiy~\cite{KMP23}.
For a more detailed discussion of the subject, we refer the reader to the work of Montgomery and Pavez-Signé~\cite{MP2023}.
We remark that the proofs of \cref{thm:d=k-2-random-robust,thm:d=k-2-counting} now simply follow by combining \cref{prp:d=k-2} with \cref{cor:random-robust,cor:counting}, respectively.
More generally, these results give optimal answers whenever \cref{con:thresholds} holds.

\subsection{Methodology}

To explain our approach, we focus on the Dirac-setting where Hamiltonicity is forced via minimum degree conditions.
All known minimum degree thresholds for Hamilton $\ell$-cycles have thus far been obtained by a combination of the Absorption Method and the Regularity Lemma.
The general approach has two phases.
First, one covers most vertices with a very large $\ell$-cycle $C$; then, one integrates the leftover vertices into $C$.
The first part is a typical application of the Regularity Lemma, while the second part is usually carried out by means of the Absorption Method (although there are exceptions~\cite{KO06}).
We briefly explain these approaches, and why they inadequate in this context.

The Absorption Method in its modern form was introduced in the seminal work of Rödl, Ruciński and Szemerédi~\cite{RRS06} on tight Hamilton cycles. 
Roughly speaking, the Absorption Method allows us to integrate the leftover vertices using small devices called `absorbers'.
While it is a powerful technique, the main obstacle to applying the Absorption Method in the context of our work is that we do not know how to construct `absorbers' when $k/2 \leq \ell \leq k-2$.
Recall from the introduction that all known $d$-degree thresholds assume that $\ell \leq k/2$.
This means that edges in $\ell$-cycles only interact with their direct neighbours, which has been a crucial property for the construction of absorbers so far.

The Regularity Lemma of Szemerédi~\cite{Sze76} was first used in the context of Dirac-type problems to find powers of cycles in the $2$-graph setting~\cite{KSS98}.
Beginning with the work of Rödl, Ruciński and Szemerédi~\cite{RRS06}, it has been applied in most results on Hamiltonicity for hypergraphs.
The idea is to approximate a given Dirac-type $k$-graph $G$ with a quasirandom blow-up $R(\cV)$ such that $\cV=\{V_x\}_{x \in V(R)}$ partitions $V(G)$ and $R(\cV)$ is obtained by replacing the edges of the $k$-graph $R$ with  quasirandom $k$-partite $k$-graphs on the corresponding clusters of $\cV$.
Moreover, $R$ has bounded order and approximately inherits the minimum degree conditions of $G$.
One then turns the solution to a (fractional) perfect tiling problem in $R$ into a collection of $\ell$-paths $\cP \subset G$ that cover most of~$V(G)$.
The paths $\cP$ are consequently connected up to the aforementioned $\ell$-cycle $C$.
The main drawback of this method is that the `connecting step' can become quite technical in the setting of hypergraphs.
Although additional machinery has been developed to ease this process~\cite{ABCM17}, the connectivity problems that arise from combinations of the Regularity Lemma and the Absorption Method when $k/2 \leq \ell \leq k-1$ remain a considerable (if not prohibitive) technical burden~\cite{LS22}.

Given these difficulties, the proof of our main result (\cref{thm:framework-hamilton-cycle}) is instead based on a recently introduced approach using blow-up covers~\cite{Lan23,LS24a}.
The idea is to find complete blow-ups $R^1(\cV^1),\dots,R^t(\cV^t) \subset G$ such that the set families $\cV^i=\{V_x^i\}_{x \in V(R^i)}$ together cover all vertices of~$G$ and $R^i(\cV^i)$ is obtained by replacing the edges of the $k$-graph $R^i$ with  complete (as opposed to quasirandom) partite  $k$-graphs on the corresponding clusters of $\cV^i$.
As before, each $R^i$ should be of bounded order and approximately inherit the minimum degree conditions of $G^i$.
In our work on tight Hamilton cycles, we proved that such a setup exists under the (suitably processed) assumptions of \cref{thm:framework-hamilton-cycle} (see~\cref{prp:blow-up-cover}).
While the families $\cV^1,\dots,\cV^t$ cover individually few vertices, we can guarantee that they overlap in a cyclical way.
This reduces the problem of finding a Hamilton $\ell$-cycle in $G$ to finding a Hamilton $\ell$-path in each $R^i(\cV^i)$ with predetermined ending tuples, which allows us to connect the paths going around the chain of blow-ups.
We remark that we require the consistency property of \cref{def:hamilton-framework} to `synchronise' the path endings between blow-ups.

The discussion above reduces the problem of finding a Hamilton cycle in the original graph $G$ to solving the Hamilton connectedness problem in a complete blow-up $R(\cV)$.
We call this the `allocation problem'.
At this point, we may assume that $R$ satisfies the connectivity and space properties of \cref{def:hamilton-framework}.
Moreover, the clusters of $\cV$ are balanced and much larger than the order of $R$.
(In practice, $\cV$ also contains one exceptional singleton cluster, a detail that we ignore here.)
We first construct a short path $P$ that starts and ends in the given tuples and visits the clusters of every edge of $R$ in every possible orientation, which is possible by combining the connectedness of $R$ with the fact that $k-\ell \not \mid k$.
Next, we consider a perfect (integer-valued) $\ell$-cycle tiling $\cT \subset R(\cV) - V(P)$.
Assuming that such a tiling exists, we may finish the proof as follows.
Since $R(\cV)$ is complete, this allows us to extend $P$ to a Hamilton path using a number of copies of each $\ell$-cycle $F \subset R$ that is indicated by $\cT$.
It remains to find a perfect $\ell$-cycle tiling in $R(\cV)-V(P)$.
This is done by showing that $R$ satisfies certain divisibility properties, leading to its lattice being complete (see \cref{sec:lattice-completeness}).
We remark that at this point, it is again crucial that $k-\ell \not \mid k$.
Together with the space property of $R$, we obtain the desired perfect $\ell$-cycle tiling of $R(\cV)-P$ as desired.

\section{Applications}\label{sec:applications-proofs}

In this section, we prove \cref{thm:d=k-1,thm:d=k-2,thm:cycle-threshold=tiling-treshold}.
Given \cref{cor:framework-simple}, it suffices to verify the connectivity, space and consistency properties of \cref{def:hamilton-framework} in each of the cases.

\subsection{Connectivity}

To show \cref{thm:cycle-threshold=tiling-treshold}, we need the following observation on connectivity.

\begin{lemma}\label{lem:connecitivty}
	For $\ell < d$, every $k$-graph with positive minimum $d$-degree is $\ell$-connected.
\end{lemma}
\begin{proof}
	Let $G$ be a $k$-graph with $\delta_d(G) > 0$, and $S \subset V(G)$ be a $d$-set.
	Since $\ell \leq d$, the edges of $G$ that contain $S$ are in the same $\ell$-component, which we denote by $C_S$.
	Now let $S' \subset V(G)$ be another $d$-set and suppose for the sake of contradiction that $C_S \neq C_{S'}$.
	Moreover, assume that $|S \cap S'|$ is maximal with this property.
	Since $C_S \neq C_{S'}$, we must have $|S \cap S'| < \ell < d$.
	Let $S''$ be obtained from $S'$ by removing a vertex of $S' \sm S$ and adding a vertex of $S \setminus S'$.
	Then $C_S = C_{S''}$ by maximality.
	But we also have $C_{S'} = C_{S''}$, since $|S' \cap S''| = d-1 \geq \ell$; a contradiction.
\end{proof}

For the proof of \cref{thm:d=k-2}, we need the following result on connectivity.

\begin{lemma}\label{lem:connecitivty-graphs}
	For $k\geq 3$ and $\eps > 0$, there is $n_0$ such that every $n$-vertex $k$-graph $G$ with $\delta_{k-2}(G) = (1/4 + \eps )\binom{n}{2}$ is $(k-2)$-connected.
\end{lemma}
\begin{proof}
	Consider a $(k-2)$-set $S \subset V(G)$.
	Since $\ell = k-2$, the edges of $G$ that contain $S$ are in the same $\ell$-component, which we denote by $C_S$.
	Now let $S' \subset V(G)$ be another $(k-2)$-set and suppose for the sake of contradiction that $C_S \neq C_{S'}$.
	Moreover, assume that $|S \cap S'|$ is maximal with this property.
	Let $z \in S \sm S'$ and $z' \in S' \sm S$.
	Let $L(S)$ denote the $2$-graph on $V(G)$ with an edge $X$ whenever $X \cup S$ is in $G$.
	Define $L(S')$ in the same way.
	Since $L(S)$ and $L(S')$ each contain at least $(1/4 + \eps )\binom{n}{2}$ edges,
	each of them contains more than $n/2$ non-isolated vertices.
	Hence, we can find edges $X$ and $X'$, in $L(S)$ and $L(S')$ respectively, which share a vertex $x \in X \cap X'$.
	Define $\hat S = (S \sm \{z\}) \cup \{x\}$ and $\hat S' = (S' \sm \{z'\}) \cup \{x\}$.
	By maximality, we have that $C_{\hat S} = C_{\hat S'}$.
	However, we also have $C_{S} = C_{\hat S}$ and  $C_{S'} = C_{\hat S'}$ by the choice of $x$.
	A contradiction.
\end{proof}

\subsection{Tilings}

We derive \cref{thm:d=k-1} and the first part of \cref{thm:d=k-2} from \cref{thm:cycle-threshold=tiling-treshold} by determining the fractional cycle tiling thresholds.

A \emph{proper $t$-colouring} of a $k$-graph $F$ is an assignment of $t$ natural numbers (\emph{colours}) to the vertices of $F$ such that no edge contains two vertices of the same colour.
So a proper $k$-colouring yields a partition of $V(H)$ into $k$ parts witnessing that $H$ is $k$-partite, where each part corresponds to the set of vertices receiving a given colour.
We require a few tools from Kühn, Mycroft and Osthus~\cite{KMO10}, showing that $\ell$-paths admit proper $k$-colourings of flexible shapes.
Recall that $\lambda(k, \ell)$ is defined in \eqref{itm:lamda}.
For $y \geq 0$, write $x \pm y$ to mean the numbers $z$ between $x-y$ and $x+y$.

\begin{lemma}[{\cite[Lemma 7.3]{KMO10}}]\label{lem:KMO-space}
	Let $P$ be an $\ell$-path on $m$ vertices.
	Then there is a proper $k$-colouring of $P$ with colours $1, \dots, k$ such that colour $k$ is used $ \lambda(k,\ell)m  \pm 1$ times and the sizes of all other colour classes are as equal as possible.
\end{lemma}

\begin{lemma}[{\cite[Proposition 3.1]{KMO10}}]\label{prop:KMO-strongly-connected}
	Let $k\geq 3$ and $1\leq \ell \leq k-1$ with $k-\ell \not \mid k$.
	Let $H$ be a complete $k$-partite $k$-graph with parts of size $k\ell(k -\ell) + 1$.
	Let $f_1,\,f_2$ be disjoint $\ell$-tuples with at most one vertex in each part.
	Then $H$ contains a spanning $(f_1,f_2,\ell)$-path.
\end{lemma}

\begin{corollary}\label{cor:space-cycles}
	Let $1 \leq \ell \leq k-1$ with $k-\ell \not \mid k$ 
	and $1/k, \eps \gg 1/t$.
	Then there is a $k$-partite $k$-uniform $\ell$-cycle on $m=t(k-\ell)$ vertices with parts of size $\alpha_1 m,\dots, \alpha_k m$ such that $\alpha_1 = \lambda(k,\ell) \pm \eps$ and $\alpha_2,\dots,\alpha_k = (1-\lambda(k,\ell))/(k-1) \pm \eps$.
\end{corollary}

\begin{proof}
	Let $b=(k\ell(k -\ell) + 1)k-2\ell$.
	Note that $b \equiv -\ell \bmod k-\ell$, so $t(k-\ell) - b \equiv k \bmod k-\ell$.
	Let $P$ be a $k$-uniform $\ell$-path on $t(k-\ell) - b$ vertices.
	The proper $k$-colouring of $P$ given by \cref{lem:KMO-space} yields parts $V_1, \dotsc, V_k$ of $V(P)$.
	Using \cref{prop:KMO-strongly-connected} and $b$ additional vertices, we can turn $P$ into a $k$-partite $\ell$-cycle $C$.
	By the choice of $t$, the vertices of $V(C) \setminus V(P)$ occupy a very small share of $V(C)$.
	So the part sizes of $C$ are essentially the same as those given by the partition of $P$.
	In particular, the bounds on $\alpha_1, \dots, \alpha_k$ follow for $t$ large enough.
\end{proof}

We also require the next result proved by the first author~\cite{Lan23}, which is also implicit in the work of Mycroft~\cite{Myc16}.
For a $k$-graph $F$, we define the threshold $\delta_d^{\frt}(F)$ as the infimum of $\delta \in [0,1]$ such that for every $\eps > 0$, there is $n_0$ such that every $k$-graph $G$ on $n \geq n_0$  vertices with $\delta_d(G)\geq (\delta + \eps ) \binom{n-d}{k-d}$ contains a perfect fractional $\{F\}$-tiling.

\begin{lemma}[{\cite[Proposition 5.7]{Lan23}}]\label{lem:fractional-thresholds-k-1}
	Let $F$ be a $k$-partite $m$-vertex $k$-graph with parts of size $m_1 \leq \dots \leq m_k$.
	Then $\delta_{k-1}^{\frt}(F) \leq m_1/m$.
\end{lemma}

The proof of the following result relies on the work of Grosu and Hladk\'y~\cite{GH12}.

\begin{lemma}[{\cite[Lemma 6.9]{Lan23}}]\label{lem:fractional-thresholds-k-2}
	For $k \geq 3$, let $F$ be a $k$-partite $m$-vertex $k$-graph with parts of size $m_1 \leq \dots \leq m_k$.
	Then $\delta_{k-2}^{\frt}(F) \leq \max\{1-(1-m_1/m)^2,\,((m_1+m_2)/m)^2\}$.
\end{lemma}

We obtain the following corollary, bounding $\delta_{d}^{\frct}(k, \ell)$ for $d = k-1$.

\begin{corollary}\label{cor:d=k-1}
	For $1 \leq \ell \leq k-2$ with $k-\ell \not \mid k$, we have $\delta_{k-1}^{\frct}(k, \ell) \leq \lambda(k,\ell)$.
\end{corollary}

\begin{proof}
	Given $1 \leq \ell \leq k-2$, let $\eps > 0$ be sufficiently small, and let $n_0=n_0(\eps)$ be sufficiently large.
	Let $G$ be a $k$-graph on $n \geq n_0$ vertices with $\delta_{k-1}(G) \geq (\lambda(k,\ell) + 2\eps) n$.
	By \cref{cor:space-cycles}, there is an $\ell$-cycle $C$ on $m=t\cdot (k-\ell)$ vertices with parts of size $\alpha_1 m\, \dots, \alpha_km$ such that $\alpha_1 = \lambda(k, \ell) \pm \eps$ and $\alpha_2, \dotsc, \alpha_k = (1 - \lambda(k, \ell))/(k-1) \pm \eps$.
	Recall the definition of $\lambda(k,\ell)$ in~\eqref{itm:lamda}.
	Since $k - \ell$ does not divide $k$, we obtain that $\lambda(k, \ell) < 1/k$.
	After rearranging, this gives $\lambda(k, \ell) < (1 - \lambda(k, \ell))/(k-1)$.
	Thus for $\eps$ small enough, $\alpha_1$ is the minimum over all $\alpha_i$.
	Then it follows by \cref{lem:fractional-thresholds-k-1} that $G$ has a perfect fractional  $\{C\}$-tiling.
\end{proof}

The next corollary handles the case $d = k-2$.

\begin{corollary}\label{cor:d=k-2}
	For $1 \leq \ell \leq k-2$ and $\eps > 0$, we have $\delta_{k-2}^{\frct}(k, \ell) \leq 1 - (1-\lambda(k,\ell))^2$.
\end{corollary} 

The proof is very similar to the one of \cref{cor:d=k-1}, now applying \cref{lem:fractional-thresholds-k-2} in place of \cref{lem:fractional-thresholds-k-1}.
The key technical point to apply \cref{lem:fractional-thresholds-k-2} is to verify that, for all $k \geq 3$ and $1 \leq \ell \leq k-2$, the inequality $1-(1-\lambda(k, \ell))^2 > (\lambda(k, \ell) + (1 - \lambda(k, \ell))/(k-1))^2$ holds.
This can be seen to be equivalent to $\lambda(k, \ell) \geq 1/(2k^2 - 6k + 5)$, which indeed holds for all $1 \leq \ell \leq k-2$.
We omit further details.

\subsection{Proof of the applications}
 
Now we derive \cref{thm:d=k-1,thm:d=k-2,thm:cycle-threshold=tiling-treshold} from \cref{cor:framework-simple}.

\begin{proof}[Proof of \cref{thm:cycle-threshold=tiling-treshold}]
	Given $1 \leq \ell < d \leq k-1$ with $k-\ell \not \mid k$, set $\delta = \delta_{d}^{\frct}(k, \ell)$.
	By \cref{cor:framework-simple}, it suffices to prove that $\delta_d^{\hf}(k, \ell) \leq \delta$.
	Let $\eps > 0$, and let $n$ be sufficiently large with respect to~$k$ and $1/\eps$.
	For any $k$-graph $G$, we set $F(G) = G$.
	We claim that this choice of $F$ is a Hamilton $\ell$-framework for $k$-graphs in $\DegF{d}{\delta + \eps}$.
	It suffices to check properties~\ref{itm:hf-connected}, \ref{itm:hf-matching} and \ref{itm:hf-intersecting} of \cref{def:hamilton-framework}.
	
	Let $G$ and $G'$ be $n$-vertex $k$-graphs in $\DegF{d}{\delta+ \eps}$ that differ in at most one vertex.
	Property~\ref{itm:hf-connected} follows by \cref{lem:connecitivty}.
	Property~\ref{itm:hf-matching} follows by definition of $\delta = \delta_{d}^{\frct}$, and the sufficiently large choice of $n$ with respect to $k$ and $\eps$.
	Finally, property~\ref{itm:hf-intersecting} follows since every $\ell$-set of $G$ is covered by an edge of $G$ and every $\ell$-set of $G'$ is covered by an $\ell$-set of $G'$ (and $|V(G) \cap V(G')| \geq \ell$ since $n \geq \ell +1)$.
\end{proof}

\begin{proof}[Proof of \cref{thm:d=k-1}]
	Immediate consequence of \cref{cor:d=k-1} and \cref{thm:cycle-threshold=tiling-treshold}.
\end{proof}
 
\begin{proof}[Proof of \cref{thm:d=k-2}]
	\cref{cor:d=k-2} implies that $\delta_{k-2}^{\frct}(k, \ell) \leq 1 - (1 - \lambda(k, \ell))^2$ for all $1 \leq \ell \leq k-2$.
	Together with \cref{thm:cycle-threshold=tiling-treshold} this yields the result for $\ell \leq k-3$. 
	So from now on assume that $\ell = k-2$.
	Set $\delta = \max\{ 1/4,\, 1- (1-\lambda(k,k-2))^2\}$.
	By \cref{cor:framework-simple}, it suffices to prove that $\delta_d^{\hf}(k, k-2) \leq \delta$.
	Given $\eps > 0$, let $n$ be sufficiently large with respect to $k$ and $1/\eps$. 
	For any $k$-graph~$G$, we set $F(G) = G$.
	We claim that this choice of $F$ is a Hamilton $(k-2)$-framework for $k$-graphs in $\DegF{d}{\delta + \eps}$, which suffices to finish the proof.
	
	Let $G$ and $G'$ be $n$-vertex $k$-graphs in $\DegF{d}{\delta + \eps}$ that differ in at most one vertex.
	We check properties~\ref{itm:hf-connected}, \ref{itm:hf-matching} and \ref{itm:hf-intersecting} of \cref{def:hamilton-framework}.
	Property~\ref{itm:hf-connected} follows by \cref{lem:connecitivty-graphs}.
	Property~\ref{itm:hf-matching} follows from $\delta \geq \delta_{k-2}^{\frct}(k, k-2)$ and the choice of $n$.
	Finally, property~\ref{itm:hf-intersecting} follows since every $\ell$-set of $G$ is covered by an edge of $G$ and every $\ell$-set of $G'$ is covered by an $\ell$-set of $G'$ (and $|V(G) \cap V(G')| \geq \ell$ since $n \geq \ell +1)$.
\end{proof}

\section{Proof of the main result}\label{sec:proof-main-result}

The goal of this section is to show \cref{thm:framework-hamilton-cycle}.
To this end, we first reformulate the input assumptions in terms of bounded hypergraphs.

\subsection{Preprocessing}

Recall the definitions of $k$-bounded hypergraphs in \cref{sec:connectedness-results}.
{Given a $[k]$-graph~$G$} and an edge $e \in G^{(\ell)}$ with $1 \leq \ell \leq k-1$, we denote by $\tc(e)$ the \tight $\ell$-component comprising all $k$-edges which contain $e$.
The \emph{$\ell$-adherence} $\adh_\ell(G) \subset G^{(k)}$ is obtained by taking the union of the \tight $\ell$-components $\tc(e)$ over all $e \in G^{(\ell)}$.

\begin{definition}[Connectivity]
	For $1 \leq \ell \leq k-1$, let $\dcon_\ell$ be the set of $[k]$-graphs~$G$ such that $\adh_\ell(G)$ is a single vertex-spanning \tight $\ell$-component.
\end{definition}

\begin{definition}[Space]
	For a $k$-graph $F$, let $\dspa_\ell$ be the set of $[k]$-graphs $G$ with $k\geq 2$ such that $\adh_\ell(G)$ has a perfect fractional $\ell$-cycle tiling.
\end{definition}

Next, we recover the concepts regarding subsampling for the $k$-bounded setting.

\begin{definition}[Property graph]\label{def:property-graph-digraph}
	For a $[k]$-graph $G$ and a family of $[k]$-graphs $\P$, the \emph{property graph}, denoted by $\PG{G}{\P}{s}$, is the $s$-graph on vertex set $V(G)$ with an edge $S \subset V(G)$ whenever the induced subgraph $G[S]$ \emph{satisfies}~$\P$, that is $G[S] \in \P$.
\end{definition}

The following definition of robustness corresponds to bounded hypergraphs.

\begin{definition}[Robustness]\label{pdef:robustness-detailed}
	For a family of $[k]$-graphs $\P$, a $[k]$-graph $G$ \emph{$(\delta,r,s)$-robustly satisfies}~$\P$ if the minimum $r$-degree of the property $s$-graph $\PG{G}{  \P   }{s}$ is at least $\delta  \tbinom{n-r}{s-r}$.
\end{definition}

We abbreviate $(1-1/s^2,r,s)$-robustness to \emph{$(r,s)$-robustness}.
Moreover, \emph{$s$-robustness} is short for $(r,s)$-robustness with $r=2k$ as in the context of \cref{pdef:robustness-detailed}.

We process the conditions of \cref{thm:framework-hamilton-cycle} using the next result, which we prove in \cref{sec:transition}.

{\begin{lemma}[Transition] \label{lem:transition}
	Let $1/\ell > 1/k,\, 1/s_1 \gg 1/s_2 \gg 1/n$ with $k-\ell \not \mid k$.
	Let $\cG$ be a family of $s_1$-graphs that admits a Hamilton $\ell$-framework.
	Let $G$ be a $k$-graph on $n$ vertices that $s_1$-robustly satisfies $\cG$.
	Then there is an $n$-vertex $[k]$-graph $G' \subset G \cup \partial_\ell(G)$ that $s_2$-robustly satisfies $\dcon_\ell \cap \dspa_\ell$.
\end{lemma}}

It is convenient to work with properties that survive the deletion of a few vertices.
For a $[k]$-graph family $\P$, denote by $\Del_q(\P)$ the set of graphs $H \in \P$ such that $H-X \in \P$ for every set $X \subset V(H)$ of at most $q$ vertices.
As it turns out, one can strengthen the properties ${\dcon_\ell}$ and ${\dspa_\ell}$ of \cref{lem:transition} against vertex deletion.
 
The next result is adapted from our former work~{\cite[Lemma 6.4]{LS24a}}.

\begin{lemma}[Booster]\label{lem:booster}
	Let $1/k,\, 1/s_1,\, 1/r_2, \, \eps, \, 1/q \gg 1/s_2 \gg 1/n$.
	Set $r_1=2k-2$ and $\delta_1 = {1-1/s + \eps}$ and $\delta_2 = 1-\exp(-\sqrt{s_2})$.
	Then every $[k]$-graph on $n$ vertices that $(\delta_1,r_1,s_1)$-robustly satisfies $\dcon_\ell \cap \dspa_\ell$ also $(\delta_2,r_2,s_2)$-robustly satisfies $\Del_q(\dcon_\ell \cap \dspa_\ell)$.
\end{lemma}

We prove \cref{lem:booster} in \cref{sec:boosting}.
To place \cref{lem:booster} in the context of \cref{def:robustness,pdef:robustness-detailed}, we recall the following fact due to Daykin and Häggkvist~\cite{DH81}, which follows from a simple double counting argument (see also~\cite[Lemma B.2]{LS24b}).

\begin{theorem}\label{fact:matchingthresholds}
	For $1/k,\, \eps \gg 1/n$ with $n$ divisible by $k$, every $k$-graph $G$ on $n$ vertices with $\delta_1(G) \geq (1-1/s +\eps) \binom{n-1}{k-1}$ contains a perfect matching. 
\end{theorem}

Finally, we recall the monotone behaviour of minimum degrees, which allows us to transition between different degree types.
\begin{fact}\label{fact:monotone-degrees}
	For an $n$-vertex $k$-graph $G$ and $d \leq d'$, we have ${\delta_d(G)}/{\binom{n-d}{k-d}} \geq  {\delta_{d'}(G)}/{\binom{n-d'}{k-d'}}.$
\end{fact}

{Using these observations, the following corollary of \Cref{lem:booster} is immediate from \cref{pdef:robustness-detailed} (and the definitions afterwards).}

\begin{corollary}[Booster]\label{corollary:booster2}
	Let $1/k,\, 1/s_1,\, 1/q,\, 1/r_2 \gg 1/s_2 \gg 1/n$,
	and set $\delta = 1 - \exp(-\sqrt{s_2})$.
	Every $[k]$-graph on $n$ vertices that $s_1$-robustly satisfies $\dcon_\ell \cap \dspa_\ell$ also $(\delta, r_2, s_2)$-robustly satisfies $\Del_q(\dcon_\ell \cap \dspa_\ell)$. \hfill \qedsymbol
\end{corollary}

\subsection{Covering with blow-ups}\label{sec:blow-up-cover-statments}

In the following, we introduce the main technical instrument for the proof of \cref{thm:framework-hamilton-cycle}, which allows us to cover the vertex set of the input graph with a chain of blow-ups, whose reduced graphs inherit the framework properties.

Consider a \emph{set family} $\cV$, which is, by convention of this paper, a family of pairwise disjoint sets.
A set $X$ is \emph{$\cV$-partite} if it has at most one vertex in each part of $\cV$.
We say that $\cV$ is \emph{$m$-balanced} if $|V| = m$ for  every $V \in \cV$.
Similarly, $\cV$ is \emph{$(1\pm\eta)m$-balanced} if $(1-\eta)m \leq |V| \leq  (1 + \eta) m$ for  every $V \in \cV$.
Moreover, $\cV$ is \emph{quasi $(1\pm\eta)m$-balanced} if there is an \emph{exceptional} set $V^\ast \in \cV$ with $|V^\ast|=1$ and $\cV \sm \{V^\ast\}$ is $(1\pm \eta)m$-balanced.
We say that another set family $\cW$ \emph{hits} $\cV$ with subfamily $\cW' \subset \cW$ if each part of $\cV$ contains (as a subset) exactly one part of $\cW'$.
(In particular, the parts of $\cW \sm \cW'$ are disjoint from the parts of $\cV$.)

\begin{definition}[Cover]
	Let $F$ be a $2$-graph and $V$ be a set.
	We say that $\{\cV^x\}_{x \in V(F)}$ and $\{\cW^{e}\}_{e \in F}$ form an \emph{$(s_1,s_2)$-sized $(m_1,m_2,\eta)$-balanced cover} of $V$ with \emph{shape} $F$ if the following holds:
	\begin{enumerate}[(C1)]
		\item Each $\cV^v$ is a quasi $(1\pm \eta)m_1$-balanced set family of size $s_1$.
		\item Each $\cW^e$ is an $m_2$-balanced set family of size $s_2$. Moreover, the vertex sets $\bigcup \cW^e$ and $\bigcup \cW^f$ are pairwise disjoint for distinct $e,f \in F$.
		\item For each $x \in e$, the set family $\cW^{e}$ hits $\cV^{x} \sm \{V^\ast\}$ with subfamily $\cW^{e}_x \subset \cW^{e}$, where $V^\ast$ is the (exceptional) singleton cluster of $\cV^x$.
		\item The family $\bigcup_{x \in V(F)} \cV^x \cup \bigcup_{xy \in F} \cW^{xy} \sm (\cW^{xy}_x \cup \cW_y^{xy})$ partitions~$V$.
	\end{enumerate}
\end{definition}

Consider a $[k]$-graph $R$, and let $\cV=\{V_x\}_{x \in V(R)}$ be a set family.
We write $R(\cV)$ for the \emph{blow-up} of~$R$ by $\cV$, which is a $[k]$-graph with vertex set $\bigcup \cV$ and whose edge set is the union of $V_{x_1} \times \dots \times V_{x_j}$ over all $j$-edges $x_1\dots x_\ell \in R$ with $j \in [k]$.
In this context, we refer to the sets of $\cV$ as \emph{clusters} and to $R$ as the \emph{reduced graph}.
If $R$ is in a family of $[k]$-graphs $\P$, we call $R(\cV)$ a \emph{$\P$-blow-up}.
We write $R(m)$ for a blow-up $R(\cV)$ whose underlying partition $\cV$ is $m$-balanced.
For a $[k]$-graph $G$ with $\bigcup\cV \subset V(G)$, we denote by $G[\cV]$ the graph on vertex set $\bigcup \cV$ that contains all $\cV$-partite edges of $G$.

\begin{definition}[Blow-up cover]
	Let $G$ be a $[k]$-graph.
	Let $\P$ be a family of $[k]$-graphs.
	An $(s_1,s_2)$-sized $(m_1,m_2,\eta)$-balanced cover formed by $\{\cV^v\}_{v \in V(F)}$ and $\{\cW^{e}\}_{e \in F}$ of $V(G)$ with  \emph{shape} $2$-graph $F$ is called a \emph{$\P$-cover} of $G$ if the following holds:
	\begin{enumerate}[\upshape (B1)]
		\item For each $v \in V(F)$, there is an $s_1$-vertex $R^{v} \in \P$ with $R^{v}(\cV^{v}) = G[\cV^v]$.
		\item For every $e \in F$, there is an $s_2$-vertex $R^e \in \P$ with $R^e(\cW^e) = G[\cW^e]$.
	\end{enumerate}
\end{definition}
	
The following result guarantees that $\P$-covers exist in hypergraphs which satisfy $\P$ in a sufficiently robust way; and that we can even find such covers with clusters of size polynomial in $\log \log n$.
It was proved in the context of directed hypergraphs~\cite[Proposition 6.3]{LS24a}.
We state an undirected version with a weaker parametrisation, which is an immediate corollary.

\begin{proposition}[Blow-up cover]\label{prp:blow-up-cover}
	Let $1/k,\, 1/\Delta,\,  1/s_1 \gg 1/s_2 \gg  \eta \gg 1/m_2 \gg 1/m_1 \gg 1/n$.
	Let  $G$ be an $n$-vertex $[k]$-graph.
	Let $\P$ be a family of $[k]$-graphs.
	Suppose that $G$ satisfies $\P$ both $(1,s_1)$-robustly and {$(2s_1-2,s_2)$}-robustly.
	Then $G$ has an $(s_1,s_2)$-sized $(m_1,m_2,\eta)$-balanced $\P$-cover whose shape is a $2$-uniform $1$-cycle.
\end{proposition}

We note that \cref{prp:blow-up-cover} is obtained from~\cite[Proposition 6.3]{LS24a} by considering the $[k]$-bounded directed hypergraph $\ori{G}$ obtained by adding for every edge $e$ of $G$ all orientations of~$e$.

\subsection{Allocation}

The next result implements the allocation step in the proof of \cref{thm:framework-connectedness-robust} and is proved in \Cref{sec:allocation-cycle}.
Note that we only allocate to a single blow-up. 

\begin{proposition}[Hamilton path allocation]\label{prp:allocation-Hamilton-cycle}
	Let $1/\ell, \, 1/k,\, 1/s \gg \eta \gg 1/m$ with $k-\ell \not \mid k$.
	Let $R \in \Del_k (\dcon_\ell \cap \dspa_\ell)$ be an $s$-vertex $[k]$-graph.
	Let $\cV=\{V_x\}_{x \in V(R)}$ be a quasi $(1 \pm \eta)m$-balanced partition of $n$ vertices with $n\equiv k \bmod (k-\ell)$ and exceptional cluster $V^\ast$.
	Let $f_1,f_2 \in (R(\cV)-V^\ast)^{(\ell)}$ be vertex-disjoint.
	Then $R(\cV)$ has a \tight Hamilton $(f_1,f_2,\ell)$-path.
\end{proposition}

\subsection{Putting everything together}\label{sec:proof-main-result-bandwidth}

Given these preliminaries, we are now ready to prove our main result (\Cref{thm:framework-hamilton-cycle}).

\begin{proof}[Proof of \cref{thm:framework-hamilton-cycle}]
	Given $1 \leq \ell \leq k-1$ with $k-\ell \not \mid k$ and $s$, we introduce constants $s'$, $s_1$, $s_2$, $\eta$, $m_2$, $m_1$ and $n$ such that
	\begin{align*}
		1/k,\, 1/t,\, 1/s \gg 1/s' \gg  1/s_1 \gg 1/s_2 \gg   \eta \gg 1/m_2 \gg 1/m_1 \gg 1/n 
	\end{align*}
	in accordance with \cref{prp:blow-up-cover} such that in addition $k-\ell$ divides $n$.
	Let $\cG$ be a family of $s$-vertex $k$-graphs that admits a Hamilton $\ell$-framework.
	Let $G$ be a  $k$-graph on $n$ vertices that $s$-robustly satisfies $\cG$.
	Our goal is to show that $G$ contains a Hamilton $\ell$-cycle.
	
	We abbreviate $\P = \Del_{k+1}(\dcon_\ell \cap \dspa_\ell)$.
	By \cref{lem:transition}, there is an $n$-vertex $[k]$-graph $G' \subset G \cup \partial_\ell(G)$ that $s'$-robustly satisfies $\dcon_\ell \cap \dspa_\ell$.
	Note that $G$ satisfies both $(1,s_1)$-robustly and $(2s_1-2,s_2)$-robustly $\P$ by two applications of \cref{corollary:booster2}.
	Hence, an application of \cref{prp:blow-up-cover} reveals that $G$ has an $(s_1,s_2)$-sized $(m_1,m_2,\eta)$-balanced $\P$-cover formed by $\{\cV^v\}_{v \in V(F)}$ and $\{\cW^{e}\}_{e \in F}$ whose shape $F$ is a $2$-uniform $1$-cycle.
	We denote the corresponding reduced $[k]$-graphs of $\P$ by $ R^v$ and $ R^e$ for $v \in V(F)$ and $e \in F$.
	We identify the vertices of the cycle $F$ with $V(F) = \{1,\dots,b\}$ following the natural cyclic ordering.
	
	{By deleting at most $2(k-\ell)-2$ vertices from each $\cW^{e}$ with $e \in F$, keeping the names for convenience, we may ensure that 
	\begin{enumerate}[(a)]
		\item \label{itm:div-W} $|\bigcup \cW^{f}| \equiv k \bmod k-\ell$   for each $f \in F$ and
		\item \label{itm:div-V} $|\bigcup \cV^{v}| - |\bigcup \cW^{uv}| - |\bigcup \cW^{vw}| \equiv -\ell \bmod k-\ell$ for consecutive $u,v,w \in V(F)$.
	\end{enumerate}
	Indeed, we begin with $1 \in V(F)$ and delete fewer than $k-\ell$  vertices of $\bigcup \cV^1$ from $\cW^{12}$ to ensure that $\cV^1$ satisfies part~\ref{itm:div-V}.
	Then delete fewer than $k-\ell$  vertices of $\bigcup \cV^2$ from $\cW^{12}$ to ensure $\cW^{12}$ satisfies part~\ref{itm:div-W}.
	We continue in this way moving around the cycle $F$ until parts~\ref{itm:div-W} and~\ref{itm:div-V} are satisfied for all vertices and edges of $F$.
	Note that in the final step, after ensuring part~\ref{itm:div-V} for $\cV^n$, it is not necessary to delete further vertices of $\cV^1$ from $\cW^{b1}$, as $\cW^{b1}$ automatically satisfies part~\ref{itm:div-W} because $n$ is divisible by $k-\ell$.
	We remark that this leaves the families $\cV^{v}$ still $(1\pm 2\eta)m_1$-balanced and the families $\cW^{e}$ still $(1\pm 2\eta)m_2$-balanced.}
	
	For each $i \in V(F)$, we select disjoint $\ell$-sets $e_i,f_i \in R^{i}(\cV^{i})$ with $e_1 = f_b$, where in addition we have $e_i \in R^{(i-1)i}(\cW^{(i-1)i})$ and $f_i \in R^{i(i+1)}(\cW^{i(i+1)})$.
	We then apply \cref{prp:allocation-Hamilton-cycle} to find a Hamilton $(f_i,e_{j},\ell)$-path $P^{ij}$ in $R^{ij}(\cW^{ij})$ for each $i \in [b]$ and $j=i+1$.
	We delete from $\cV^i$ the vertices already used in those paths, except for $e_i \cup f_i$, keeping the names $\cV^i$ for convenience.
	Note that the families $\cV^{i}$ are still $(1\pm 4\eta)m_1$-balanced after the deletions. 
	Also, by~\ref{itm:div-V}, since we deleted all but $|e_i \cup f_i| = 2\ell$ vertices from $\cV$, we now have that $|\bigcup \cV^i| \equiv - \ell + 2 \ell \equiv k \bmod k - \ell$.
	We may therefore finish by applying \cref{prp:allocation-Hamilton-cycle} to find a Hamilton $(e_i,f_{i},\ell)$-path $P^{i}$ in $R^{i}(\cV^{i})$ for each $i \in [b]$ with index computations taken modulo~$b$. 
\end{proof}

\section{Allocation}\label{sec:allocation-cycle}

In the following, we prove \cref{prp:allocation-Hamilton-cycle}.
We first use the assumptions to show that one can find an integral cycle tiling.
Then we find the Hamilton path.

\subsection{Lattice completeness} \label{sec:lattice-completeness}

To allocate a perfect tiling into a suitable blow-up, we borrow a few concepts from Keevash and Mycroft~\cite{KM15} following the exposition of the first author~\cite{Lan23}.
For $k$-graphs $F$ and $G$ and a homomorphism $\phi \in \Hom(F,G)$, we denote by $\vn_\phi \in \NATS^{V(G)}$ the \emph{indicator vector}, which counts the number of vertices of $F$ mapped to each $v \in V(G)$ by~$\phi$.
Formally, $\vn_\phi(v) = |\phi^{-1}(v)|$ for every $v \in V(G)$.
The \emph{$F$-lattice} of $G$ is the additive subgroup $\cL_F(G) \subset \INTS^{V(G)}$ generated by the vectors $\vn_\phi$ with $\phi \in \Hom(F,G)$.
We say that $\cL_F(G)$ is \emph{complete} if it contains all $\vecb b \in \INTS^{V(G)}$ for which $\sum_{v \in V(G)} \vecb b(v)$ is divisible by $v(F)$.
The goal of this section is to show the following result.

\begin{lemma}\label{lem:lattice-completeness}
	Let $1 \leq \ell \leq k-1$ with $k-\ell \not \mid k$, and let $R$ be an $\ell$-connected $k$-graph.
	Then there is a $k$-uniform $\ell$-cycle $C$ with $e(C) \leq k^4$ which is $k$-partite and such that $\cL_C(R)$ is complete.
\end{lemma}

Recall that a {proper $t$-colouring} of a $k$-graph $F$ is an assignment of natural numbers ({colours}) to the vertices of $F$ such that no edge contains two vertices of the same colour.
We define the \emph{chromatic number} $\chi(F)$ to be the least integer $t$ for which $F$ admits a proper $t$-colouring.
Let $\cD(F) \subset~\NATSZ$ be the union of the integers $||\phi^{-1}(1)|-|\phi^{-1}(2)||$ over all proper $\chi(F)$-colourings $\phi$ of $F$.
(Note that this is a non-standard definition of the chromatic number for hypergraphs.)
We denote by $\gcd(F)$ the greatest common divisor of the integers in $\cD(F)\sm \{0\}$ with the convention that $\gcd(F)=\infty$ if $\cD(F)=\{0\}$.
The following lemma gives a convenient sufficient condition for lattice completeness.

\begin{lemma}[{\cite[Lemma 6.4]{Lan23}}]\label{lem:connectivity-to-completeness}
	Suppose that $G$ is a $1$-connected $k$-graph, and that $F$ is a $k$-partite $k$-graph with $\gcd(F)=1$.
	Then $\cL_F(G)$ is complete.\footnote{We note that ``$\ell$-connected'' has a different meaning in the work of the first author~\cite{Lan23}. 
		That~$G$ is $1$-connected here implies that $G$ is ``$k$-connected'' in the language of the other paper, and since~$F$ is $k$-partite we have $\chi(F) = k$, from which this statement follows.}
\end{lemma}

Given \cref{lem:connectivity-to-completeness}, \cref{lem:lattice-completeness} is a direct consequence of the next result.

\begin{lemma}\label{prop:gcd-cycles}
	Let $1 \leq \ell \leq k-1$ with $k - \ell \not \mid k$, and let $m \geq \lceil k/(k-\ell) \rceil$.
	Then the $k$-uniform $\ell$-cycle $C$ with $k^2\ell+m$ edges is $k$-partite.
	Moreover, if $m = \lceil k/(k-\ell) \rceil$, then we also have $\gcd(C)=1$.
\end{lemma}

\begin{proof}
	Let $q = m \cdot (k-\ell) - k$.
	Consider a complete $k$-partite $k$-graph $H$ with parts  $V_1,\dots,V_k$ of size $k\ell \cdot(k -\ell) + 1$.
	Let $H'$ be the $k$-partite $k$-graph with clusters $V_1',\dots,V_k'$ obtained from $H$ by adding a vertex $u_j$ to cluster $V_{\ell+j}$ for each $j \in [q]$, with indices modulo $k$.
	Note that if $m = \lceil k/(k-\ell) \rceil$, then $0 < q < k - \ell$ (since $k-\ell \not \mid k$); so in this case it follows that $H'$ has two parts whose sizes differ by one. Hence $\gcd(H')=1$.
	
	Observe that the number of vertices of $C$ (as defined in the statement)
	is precisely $(k^2 \ell + m)(k-\ell) = k\cdot(k\ell\cdot (k-\ell) + 1) + q = v(H')$.
	Hence, to conclude it suffices to show that $H'$ contains a Hamilton $\ell$-cycle.
	Consider $f_1=(w_1,\dots,w_\ell)$ with $w_i \in V_{\ell+q+i}$ for $i \in [\ell]$ (index computations modulo $k$).
	Let $f_2=(v_1,\dots,v_\ell)$ with $v_i \in V_{i}$ and $v_i \notin f_1$ for each $i \in [\ell]$.
	By \cref{prop:KMO-strongly-connected}, $H$ contains an $(f_1,f_2,\ell)$-path $P$ on vertex set $V(H)$.
	We obtain a Hamilton $\ell$-cycle in $H'$ by appending the vertices $u_1,\,\dots,\,u_q$ to~$P$.
\end{proof}

\subsection{Perfect tilings}

In the following, we allocate a perfect tiling in the context of \cref{prp:allocation-Hamilton-cycle}.

\begin{proposition}[Perfect tiling allocation]\label{lem:perfect-tiling-allocation}
	Let $1/k,\, 1/s \gg \eta \gg 1/m$ with $k-\ell \not \mid k$.
	Let $R \in \Del_k (\dcon_\ell \cap \dspa_\ell)$ be an $s$-vertex $[k]$-graph.
	Let $\cV=\{V_x\}_{x \in V(R)}$ be a $(1 \pm \eta)m$-balanced partition with $|\bigcup \cV|$ divisible by $k-\ell$.
	Then the blow-up $R(\cV)$ has a perfect $\ell$-cycle tiling.
\end{proposition}

We need a simple auxiliary result about perfect matchings.

\begin{observation}\label{lem:balancing-matching}
	For $1/k \gg \eta \gg 1/m$, let  $G$ be a $(1\pm\eta)m$-balanced blow-up of a complete $k$-graph of order $k+1$, whose total number of vertices is divisible by $k$.
	Then $G$ has a perfect matching.\qed
\end{observation}

\COMMENT{\begin{proof}
		Let $\cV$ be the partition of $V(G)$.
		Note that we can easily find a perfect fractional matching if all clusters have the same size.
		Simply give each edge the same weight.
		Otherwise, we proceed as follows.
		Denote by $d$ the largest differences in size of two clusters.
		So $d \leq 2\eta n$ by assumption.
		To balance the clusters of $\cV$, we construct a matching $M$ as follows.
		We begin with $M = \es$.
		While there is a cluster $V \in \cV$ that is larger than all others in $G - V(M)$, select edges $e_1,\dots,e_{k}$ such that each of these edges covers a vertex of $V$ and misses a unique cluster of the others.
		So after adding these edges, the largest difference between the size of two clusters must have gone down by $1$.
		This process stops after $d$ steps, since it never uses more than $(k-1)d$ vertices of each cluster.
		We can then finish by covering $G-V(M)$ with a perfect fractional matching as describe before.
\end{proof}}

Finally, we need the following  lemma that bounds the lengths of cycles in tilings.
Let us call a fractional $\ell$-cycle tiling \emph{$m$-bounded} if it only involves cycles on at most $m$ vertices.

\begin{lemma}\label{obs:boundedn-order}
	Let $R$ be an $s$-vertex $k$-graph, which admits a perfect fractional $\ell$-cycle tiling.
	Then $R$ also has a $k^2s^{\ell}$-bounded fractional $\ell$-cycle tiling.
\end{lemma}

\begin{proof}
	Let $\cF$ be the family of all $k$-uniform $\ell$-cycles, and let $\omega \colon \Hom(\cF,R) \to [0,1]$ be a perfect fractional $\cF$-tiling.
	Suppose that $\omega(\phi) > 0$ for some homomorphism $\phi\colon V(C) \rightarrow R$, where $C$ is a cycle with $m=v(C) > k^2 s^{\ell}$.
	Denote the vertex ordering of $C$ by $v_1,\dots,v_m$.
	For each $1 \leq i \leq m$ let $X_i = \{v_{i}, \dotsc, v_{i+\ell}\}$, with indices modulo $m$.
	Note that there are at least $m/k \geq k s^\ell$ many sets $X_i$ with $\deg_C(X_i) = 2$.
	On the other hand, there are at most $s^\ell$ ordered $\ell$-sets in~$V(R)$.
	By the pigeonhole principle we can find two disjoint ordered $\ell$-sets $X_i, X_j \subset V(C)$ with $\deg_C(X_i)=\deg_C(X_j)=2$, and such that $\phi(x_{j+q}) = \phi(x_{i+q})$ for all $1 \leq q \leq \ell$.
	We can use this to find two shorter $\ell$-cycles $C'$ and $C''$, together with respective homomorphisms $\phi', \phi''$ and a perfect fractional $\cF$-tiling $\omega'$ such that $\omega'(\phi') = \omega'(\phi'') = \omega(\phi)$ and $\omega'(\phi)=0$.
	Iterating this gives the desired fractional tiling.
\end{proof}

\begin{proof}[Proof of \cref{lem:perfect-tiling-allocation}]
	Introduce $q$ with {$1/k ,\, 1/s \gg 1/q \gg \eta$}.
	For a $k$-graph $F$, $\phi \in \Hom(F,R)$ and injective $\theta \in \Hom(F,R(\cV))$, we say that the copy $\theta (F) \subset R(\cV)$ of $F$ is \emph{$\phi$-partite} if $\theta(v) \in V_{\phi(v)}$ for every vertex $v\in V(F)$.
	By \cref{lem:lattice-completeness}, there is a $k$-partite $k$-uniform $\ell$-cycle $D$ with $v(D) \leq k^4$ such that the lattice $\cL_{D}(R)$ is complete.
	We first pick a $\{D\}$-tiling $\cT_{\res} \subset R(\cV)$ that acts as a `reservoir' by selecting $q$ disjoint $\phi$-{partite} copies of $D$ for every $\phi \in \Hom(D,R)$.
	Next, we use \cref{prop:gcd-cycles} to pick an $\ell$-cycle tiling $\cT_{\divn} \subset R(\cV)$ that contains $s^s$ copies of a cycle on $k^2\ell+p$ edges for each $k \leq p \leq s^s$. 
	Our plan is to use these cycles to adjust the divisibility of the number of leftover vertices later on.
	Let $\cV'$ be obtained from $\cV$ by deleting the vertices of $\cT_{\res} \cup \cT_{\divn}$.
	Note that $\cV'$ is still $(1\pm 2\eta)m$-balanced.
	
	Let $\mathcal{C}_{s}$ be the family of $\ell$-cycles with at most $k^2 s^\ell$ vertices.
	Next, we cover most of the vertices of~$R(\cV')$ with a $\mathcal{C}_{s}$-tiling $\cT_{\cover}$.
	We would like to find such a tiling by ``blowing up'' a fractional cycle tiling in $R$, but this does not work since $\cV'$ is not perfectly balanced; we deal with this issue as follows.
	First, we pick a set $X \subseteq \bigcup \cV'$ of size at most $s-2$, so the remaining number of vertices, after removing $X$, is divisible by $s-1$.
	Let $\cV''=\{V_x''\}_{x \in V(R)}$ be the resulting partition, and note that it is still $(1\pm 3\eta)m$-balanced.
	Let $K$ be the complete $(s-1)$-uniform graph on vertex set $V(R)$.
	By applying \cref{lem:balancing-matching} to $K(\cV'')$, we find a perfect matching $M \subseteq K(\cV'')$.
	For each $x \in V(R)$, let~$m_x$ be the number of edges in $M$ not intersecting the cluster $V_x'' \in \cV''$.
	We can partition $\bigcup \cV''$ into sets $\{S_x \colon x \in V(R)\}$, where each $S_x$ intersects each cluster $V''_y$ in $m_x$ vertices if $y \neq x$, and does not intersect $V''_x$ at all.
	Then the sets $S_x$ induce an $m_x$-balanced partition $\mathcal{S}_x$ whose clusters correspond to the vertices of $R-x$; and the vertices in all partitions cover all of $\bigcup \mathcal{V}'$.
	Since $R \in \Del_1 (\dspa_\ell)$ and by \cref{obs:boundedn-order}, for each $1 \leq x \leq s$ there is a perfect fractional $\mathcal{C}_{s}$-tiling $\omega_x$ of $(R - x)(\mathcal{S}_x)$.
	It follows that $\omega = \omega_1 + \dots + \omega_s$ is a perfect fractional $\mathcal{C}_{s}$-tiling of $R(\cV'')$.
	
	Now let $C \in \mathcal{C}_{s}$, and let $\phi \in \Hom(C, R)$ be a homomorphism.
	We set $w_\phi$ to be the sum of~$\omega(F)$ over all $\phi$-partite copies $F \subset R(\cV'')$ of $C$.
	We obtain the $\ell$-cycle tiling $\cT_{\cover} \subset R(\cV'')$ by selecting~$\lfloor w_\phi  \rfloor$ many vertex-disjoint $\phi$-partite copies of the $\ell$-cycle $C$, for each possible choice of $\phi \in \Hom(\cC, R)$.
	Note that here we use the fact that $m$ is much larger than $s$ and $k$, and each $C$ has at most $k^2 s^\ell$ vertices.
	Let~$\cV'''$ be obtained from $\cV'$ by deleting the vertices of $\cT_{\cover}$ and adding the vertices of $X = \bigcup \cV' \setminus \bigcup \cV''$.
	So~$\cV'''$ now corresponds to the set of vertices of $\cV$ not yet covered by $\ell$-cycles.
	Note that $\cV'''$ contains at most $|\Hom(\cC,R)| + |X| \leq (k^2 s^\ell)^s + s-2 \leq s^{2 \ell s}$ vertices, where the error arises  from rounding $\lfloor w_\phi \rfloor$ at each $\phi \in \Hom(\cC,R)$.
	Lastly, we ensure that the number of vertices in~$\cV'''$ is divisible by~$v(D)$.
	Since the initial number of vertices was divisible by $k-\ell$, each $\ell$-cycle also has a number of vertices divisible by $k-\ell$ and $v(D) \leq k^4$, we can achieve this just by removing at most $k^5$ cycles of $\cT_{\divn}$.
	(For convenience, we keep the names $\cT_{\divn}$ and $\cV'''$.)
	Using a crude bound, we see that after this change $\cV'''$ contains at most $2s^{2 \ell s}$ vertices.
	
	To finish, we construct a $\{D\}$-tiling $\cT_{\ominus} \subset R(\cV)$ that contains the remaining vertices.
	This is done by picking an $\ell$-cycle tiling $\cT_{\oplus} \subset \cT_{\res}$ from the reservoir such that the vertices $V(\cT_{\oplus}) \cup \bigcup \cV'''$ span a perfect $\ell$-cycle tiling $\cT_{\ominus}$ in $R(\cV)$.
	Formally, let $\mathbf {b} \in \INTS^{V(R)}$ count the size of each part of $\cV'''$, that is $\mathbf {b} (x) = |V_x'''|$ for every $x \in V(R)$.
	Then $\sum_{x \in  V(R)} \mathbf {b}(x)$ is divisible by~$v(D)$ by construction of $\cT_{\cover}$.
	{It follows that $\mathbf {b}\in \cL_D(R)$.}
	By the choice of $D$, the lattice $\cL_{D}(R)$ is complete, so there are integers~$c_\phi$, one for every $\phi\in \Hom(D,R)$, such that $\mathbf {b} = \sum_{\phi \in \Hom(D,R)} c_\phi \vn_\phi$.
	Since $1/k, 1/s \gg 1/q$, we can ensure $|c_\phi| \leq q$ holds for each $\phi\in \Hom(D,R)$.
	Let $\Phi_{\oplus}$ and $\Phi_{\ominus}$ denote the sets of homomorphisms $\phi\in \Hom(D,R)$ for which~$c_\phi$ is positive and negative, respectively.
	
	Let $\cT_{\ominus} \subset \cT_{\res}$ be a subtiling obtained by selecting $|c_\phi|$ many $\phi$-partite copies of $D$ from $\cT_{\res}$ for each $\phi \in \Phi_{\ominus}$. Since
	\begin{align*}
		\sum_{\phi \in \Phi_{\oplus}} c_\phi \vn_\phi  = \mathbf {b} - \sum_{\phi \in \Phi_{\ominus}} c_\phi \vn_\phi = \mathbf {b} + \sum_{\phi \in \Phi_{\ominus}} |c_\phi| \vn_\phi,
	\end{align*}
	we may finish by selecting a $\{D\}$-tiling $\cT_{\oplus} \subset R(\cV)$ on the vertices of $V(\cT_{\ominus}) \cup \bigcup \cV''$ by adding $|c_\phi|$ many $\phi$-partite copies of $D$ for every  $\phi \in \Phi_{\oplus}$.
\end{proof}

\subsection{Hamilton cycles}

The following two simple lemmata finish our preparations to embed Hamilton paths into suitable blow-ups.

For a $k$-graph $G$, an $\ell$-cycle $C$ 
and a homomorphism $\phi\colon V(C) \to G$, we call the image $\phi(C) \subset G$ a \emph{closed $\ell$-walk}.
The \emph{length} of the closed $\ell$-walk $\phi(C)$ is the number of edges of $C$.
Given a $k$-graph~$G$, we say that an $\ell$-tuple $(v_1,\dots,v_\ell)$ is \emph{supported} in $G$ if $\{v_1,\dots,v_\ell\} \subset e$ for some edge $e \in G$.

\begin{lemma}\label{lem:strong-connectivity}
	For $1 \leq \ell \leq k-1$ with $k - \ell$ not dividing $k$, let $G$ be a $k$-graph and $C$ be an $\ell$-component of~$G$.
	Then any two $\ell$-tuples supported in $C$ are on a common closed $\ell$-walk in $G$ of length at most $k^4 e(C)$.
\end{lemma}

\begin{proof}
	By \cref{prop:KMO-strongly-connected} this is true if $C$ consists of a single edge.
	The general case follows by a simple induction on the number of edges in $C$.
\end{proof}

\begin{lemma}\label{lem:cover}
	Suppose that $R$ is a $k$-graph and $x \in V(R)$ such that $R, R-x \in \dcon_\ell$.
	Let $T$ be an $\ell$-path with $2k < v(T) \leq 3k$, and let $y \in V(T)$ be not on the endtuples of $T$.
	Then there is a $\phi \in \Hom(T,R)$ such that $\phi^{-1}(x)=\{y\}$.
\end{lemma}

\begin{proof}
	We first show that there are edges~$e,f$ in $R$ that intersect in at least $\ell$ vertices such that $x \in f$ and $x \notin e$.
	By assumption (since a $k$-graph in $\dcon_\ell$ must contain at least one edge) there are distinct edges $e',f' \in R$ such that $x \in f'$ and $x \notin e'$.
	Moreover, there is a closed $\ell$-walk $W$ containing both $f'$ and $e'$, by \cref{lem:strong-connectivity}.
	We can then take $f$ to be the last edge of $W$ that contains $x$ when going from $x$ to a vertex in $e$ along $W$, and $e$ to be the edge that follows right after $f$ in $W$.
	
	Let $f = xv_2 \dotsb v_k$ and suppose that $v_i,\dots,v_k \subset e$ for some $i \leq k-\ell+1$.
	We then start the sequence of our closed $\ell$-walk with $v_k,\dots,v_2,x,v_2,\dots,v_k$.
	We continue this sequence on both ends with vertices in $e$ until it corresponds to the image of $T$ under a homomorphism $\phi \in \Hom(T,R)$.
\end{proof}

\begin{proof}[Proof of \rf{prp:allocation-Hamilton-cycle}]
	Let $f_1,f_2 \in (R(\cV)-V^\ast)^{(\ell)}$ be vertex-disjoint.
	We have to show that $R(\cV)$ contains a \tight Hamilton $(f_1,f_2,\ell)$-path.
	By assumption, all clusters of $\cV$ but $V^\ast$ have size $(1\pm \eta)m$, and $V^\ast$ has size $1$.
	
	Let us write $V^\ast = V_{{x^\ast}} = \{v^\ast\}$ with $x^\ast \in V(R)$.
	{For every oriented $k$-edge $e = (v_1, \dotsc, v_k)$ in $V(R-x^\ast)$,
	take a $k$-vertex $\ell$-path $P_e \in R(\cV)$, consisting of a single edge.}
	Since $m$ is much larger than $k$ and $s$, we can assume that these paths are pairwise disjoint and also disjoint of $f_1,f_2$.
	
	We handle the exceptional vertex $v^\ast$ as follows.
	Use \cref{lem:cover} to pick an $\ell$-path in $R(\cV-V^\ast)$ of order at most $3k\leq k^4$ that contains $v^\ast$.
	We can ensure that~$P^\ast$ is disjoint from the paths of type $P_e,f_1,f_2$ specified above.

	{Since $R \in \Del_k(\con_\ell)$, the $k$-graphs $\adh_\ell(R)$ and $\adh_\ell(R-{x^\ast})$ are \tightly $\ell$-connected.}
	By \cref{lem:strong-connectivity}, any two $\ell$-tuples of $R-{x^\ast}$ are connected by an \tight $\ell$-walk~$W$ of order at most $k^4e(R) \leq s^{2k}$.
	This allows us to connect $f_1$ and $f_2$, together with all the $\ell$-paths $P_e$, and also $P_{x^\ast}$, to one `skeleton' $\ell$-path $P_{\text{skel}}$, which begins with $f_1$ and ends with $f_2$.
	Since we need to add at most $e(R)+3 \leq s^k$ walks to connect everything,  $P_{\text{skel}}$ has at most, say, $8^{s^k}$ vertices.
 
	Let $\cV'$ be obtained from $\cV$ by deleting the vertices of $P_{\text{skel}}$ (including the now empty exceptional cluster).
	Note that $\cV'$ is still $(1\pm 2\eta)m$-balanced.
	Since $P_{\text{skel}}$ is an $\ell$-path, its number of vertices is congruent to $k$ modulo $k - \ell$, and the same is true of $\cV$ by assumption.
	Hence the total remaining number of vertices in $\cV'$ is divisible by $k - \ell$.
	By \cref{lem:perfect-tiling-allocation} and the assumption that $R \in \Del_k (\dspa_\ell)$, there is a perfect $\ell$-cycle tiling $\cT$ of $(R-x^\ast)(\cV')$.
	
	To finish, we incorporate the cycles of $\cT$ into the skeleton path, one at a time.
	Take a cycle $C \in \cT$ not yet incorporated into the path, and let $e \in E(C)$ be any of its oriented edges.
	Suppose that $C = v_1 v_2 \dotsb v_r$, where the first $k$ vertices correspond precisely to $e$.
	By construction, we can write $P_{\text{skel}} = f_1 \dotsb y_1 y_2 \dotsb y_k \dotsb f_2$, where $y_1 y_2 \dotsb y_k$ are the vertices of the $\ell$-path $P_e$.
	Then we can insert $C$ in the middle by considering the $\ell$-path
	\[ f_1 \dotsb y_1 y_2 \dotsb y_{k-\ell} v_{k-\ell+1} \dotsb v_k v_1 v_2 \dotsb v_{k-\ell} y_{k - \ell + 1} \dotsb y_k \dotsb f_2, \]
	which covers precisely the vertices of $P_{\text{skel}}$ and $C$ and is still an $\ell$-path going from $f_1$ to $f_2$.
	Note that this new path still contains an oriented edge for every possible oriented edge of $R$.
	This means we can iterate this argument over all the cycles of $\cT$, incorporating one cycle at a time, to obtain the desired Hamilton $\ell$-path.
\end{proof}

\section{Reducing to hypergraphs of smaller uniformity}\label{sec:squashing}

In this section, we prove \cref{thm:generalsquash,cor:squash}.
Let us begin with the latter.

\begin{proof}[Proof of \cref{cor:squash}]
	Recall that an $h$-cycle with $k-h$ dividing $k$ contains (if its number of vertices is divisible by $k$) a perfect matching as a subgraph.
	Hence, the lower bound follows by considering a $qk$-graph without a perfect matching and $q\ell$-degree close to $1/2$~\cite{HPS09}.
	
	For the upper bound, note that $\th_{d}^{\ham}(qk,q\ell) \leq \th_{q(k-1)}^{\ham}(qk, q\ell) \leq \th_{k-1}^{\ham}(k,\ell) \leq 1/2$ where the first bound follows from monotonicity (\cref{fact:monotone-degrees}), the second inequality follows from \cref{thm:generalsquash}, and the last equality corresponds to known bounds for codegree thresholds for tight cycles~\cite{RRS08a} (as a tight cycle of order divisible by $k-\ell$ contains a spanning loose cycle).
\end{proof}

It remains to show \cref{thm:generalsquash}.
Suppose $k, n \geq 1$ and $q \geq 2$ are given.
Let $H$ be a $qk$-uniform hypergraph on $qn$ vertices.
Given a partition $\cQ = \{ B_1, \dotsc, B_n \}$ of $V(H)$ into sets of size $q$ each, we let the \emph{squashed hypergraph} $H_{\cQ}$ be the $k$-graph on vertex set $\{ 1, \dotsc, n \}$ where $I \in H_{\cQ}$ if and only if $\bigcup_{i \in I} B_i \in H$.
We show that random squashing of hypergraphs preserve densities with high probability.

\begin{lemma} \label{lemma:randomsquash}
	Let $k, n \geq 1$, $\eps > 0$ and $q \geq 2$ be given, let $H$ be a $qk$-uniform hypergraph on $qn$ vertices, and let $\cQ$ be a partition of $V(H)$ into sets of size $q$ chosen uniformly at random.
	Then, with probability at least $1 - 2\exp(- \eps^2 n/(16q))$, we have
	\( |H_{\cQ}| \geq  \big({\binom{n}{k}}/{\binom{qn}{qk}}\big) |H| - \eps n^k. \)
\end{lemma}

We will need a concentration inequality for random permutations~\cite[Proposition 2.3]{AlonDefantKravitz2022}.

\begin{proposition}  \label{proposition:concentration-permutation}
	Let $f: S_n \rightarrow \mathbb{R}$ be a function on permutations such that if $\pi, \pi' \in S_n$ differ by a transposition, then $|f(\pi) - f(\pi')| \leq z$.
	Then, for $\pi \in S_n$ chosen uniformly at random, we have $\Pr[ |f(\pi) - \Exp[f]| \geq t ] \leq 2 \exp( - t^2 / (4 z^2 n))$.
\end{proposition}

\begin{proof}[Proof of Lemma~\ref{lemma:randomsquash}]
	We will generate a random partition $\cQ$ of $V$ into sets of size $q$ by sampling a random permutation $\sigma \in S_{qn}$ and considering the random ordering $v_{\sigma(1)}, \dotsc, v_{\sigma(qn)}$ of $V(H)$.
	Then we define, for all $0 \leq i \leq n-1$, the sets $B_i = \{ v_{\sigma(qi + 1)}, v_{\sigma(qi+2)} \dotsc, v_{\sigma(qi + q)} \}$.
	
	First, we shall prove that $\Exp[|H_\cQ|] = \binom{n}{k} |H|/\binom{qn}{qk}$.
	By linearity of expectation it is enough to understand the probability that a given edge $e \in H$ contributes to $H_\cQ$.
	Specifically, it is enough to show that a given edge $e \in H$ is of the form $\bigcup_{i \in I} B_i$ for some $k$-set $I$ with probability $\binom{n}{k} / \binom{qn}{qk}$.
	We can see this as follows.
	For any $e \in H$, clearly the set $\sigma(e) = \{ \sigma(v) \colon v \in e\}$ corresponds to a uniformly random $qk$-set of $H$; so for a fixed $qk$-set $X$ the event $\sigma(e) = X$ has probability exactly~$1/\binom{qn}{qk}$.
	There are $\binom{n}{k}$ possible choices for $X$ (each corresponding to a choice of $k$-set of indices $I$) and for each different $X$ the events $\sigma(e) = X$ are disjoint, so indeed we see that $e$ contributes with an edge to $H_\cQ$ with probability $\binom{n}{k}/\binom{qn}{qk}$, as desired.
	
	Next, note that changing $\sigma$ by swapping two indices $i, j$ can change $|H_\cQ|$ by at most $2 n^{k-1}$ (corresponding to the edges which contain indices $i, j$ in $H_\cQ$).
	Thus, by Proposition~\ref{proposition:concentration-permutation} applied with $\eps n^k$, $qn$, $2n^{k-1}$ playing the rôles of $t$, $n$ and $z$, we get
	\begin{align*}
		\Pr \left[ |H_{\cQ}| \geq \frac{\binom{n}{k}}{\binom{qn}{qk}} |H| - \eps n^k \right] \leq 2 \exp \left( - \frac{\eps^2 n}{16q} \right),
	\end{align*}
	as required.
\end{proof}

Now we are ready to prove the main result of this section.

\begin{proof}[Proof of \cref{thm:generalsquash}]
	Let $\mu = \th_{d}^{\ham}(k,\ell)$, and let $H$ be a $qn$-vertex $qk$-graph with $\delta_{qd}(H) \geq (\mu + 2 \eps) \binom{qn - q d}{qk - qd}$.
	Let $\cQ$ be a random partition of $V(H)$ into sets of size $q$ each, and let $H_\cQ$ be the squashed $k$-graph generated by $\cQ$.
	
	For any $k$-uniform hypergraph $H$ and set $X \subseteq V(H)$, the link graph $L_H(X)$ is the $(k-|X|)$-uniform graph on $V(H) - X$ whose sets are precisely those $Y$ such that $Y \cup X \in E(H)$.
	Note that for any $d$-set $X$ in $V(H_\cQ)$, the link graph $L_{H_\cQ}(X)$ corresponds to a random squashing of the link graph of a $qd$-set in $H$, which by assumption contains at least $(\mu + 2 \eps) \binom{qn - q d}{qk - qd}$ many $q(k-d)$-sets.
	Thus, by Lemma~\ref{lemma:randomsquash}, we see that $L_{H_\cQ}(X)$ has at least $(\mu + \eps) \binom{n-d}{k-d}$ edges with probability at least $1 - 2 \exp(- \eps^2 (n-d) / 16q)$.
	Taking a union bound over all $\binom{qn}{qd}$ $qd$-sets in $V(H)$, we find that $H_\cQ$ has $\delta_d(H_\cQ) \geq (\mu + \eps) \binom{n-d}{k-d}$ with non-zero probability.
	Fixing such a choice, we deduce, since $\mu = \th_{d}^{\ham}(k,\ell)$ and $n$ is large, that $H_\cQ$ contains a Hamilton $\ell$-cycle.
	This corresponds naturally to a Hamilton $q \ell$-cycle in $H$, and we are done.
\end{proof}

\section{Conclusion}\label{sec:conclusion}

In this paper, we introduced a general framework for finding $k$-uniform Hamilton $\ell$-cycles when $k-\ell$ does not divide $k$.
As a consequence, we recover many known bounds in the Dirac setting and determine several new ones.
An important aspect of the $k$-graphs in our applications is that the $\ell$-components span all edges.
In other words, the graphs are globally connected, which allows us to take $F(G) = G$ in the context of \cref{def:hamilton-framework}.
In general, this may not be the case, and we expect that one has to select a proper subgraph $F(G) \subsetneq G$.
Our main result has been explicitly formulated to account for this situation and should therefore pave the way for future work in this direction.

An interesting next case to consider is the case for odd $k$, with $d = k-3$ and $\ell = k-2$.
The `space barrier' bound~\cite[Proposition 2.1]{HZ16} in this case is $\delta^{\ham}_{k-3}(k,k-2) \geq 1 - (1 -  {1}/{(2 \lceil k/2 \rceil )})^3$, and moreover Han--Zhao~\cite[Theorem 1.5]{HZ16} considered connectivity arguments to show that $\delta^{\ham}_{k-3}(k,k-2) \geq 1/4$.
The latter bound is better than the first for $k \geq 11$. We believe that the maximum of these two bounds should be the truth.
Beyond this, it would be interesting to obtain more accurate non-optimal bounds in terms of $k$, $d$ and $\ell$, similar to our work on tight cycles, where we proved that $\delta^{\ham}_d(k,\ell) \leq 1/(2(k-d))$ for all $1 \leq d \leq k-1$~\cite{LS22}.

Another problem concerns the situation when $k-\ell$ divides $k$.
Here our framework does not apply.
The main reason is that connectivity becomes more delicate and is no longer well-captured by a simple intersection property such as $\ell$-connectivity.
For instance in a $4$-uniform graph with an edge $\{a,b,c,d\}$, we can find a $1$-walk between any two vertices (\cref{prop:KMO-strongly-connected}).
However, we are not guaranteed a $2$-walk starting with $ab$ and ending with~$bc$.
It is not obvious whether this obstacle can be overcome in the context of Dirac-type problems (for instance by finding some gadget that gives `strong connectivity') or whether one has to restrict the search for Hamilton cycles to certain oriented $\ell$-sets.
In the latter case, it seems that the right way to tackle this problem is to work in the setting of directed hypergraphs as in our setup for tight cycles~\cite{LS24a}.
We recently learned about the related work of Letzter and Ranganathan~\cite{LR25} on positive codegree thresholds, who use blow-up covers in conjunction with tight connectivity to tackle situations when $k-\ell$ divides $k$.
 
While our applications only address the minimum degree setting, \cref{thm:framework-hamilton-cycle} applies to all properties that are preserved under subsampling.
This includes notions of quasirandomness~\cite{Lan23} and other degree conditions such as degree sequences~\cite{JLS23}.
It would be interesting to see whether one can establish Hamilton frameworks under these conditions.
Lastly, we mention that \cref{thm:framework-hamilton-cycle} also simplifies the stability analysis that is typically employed to obtain exact minimum degree conditions for Hamilton $\ell$-cycles~\cite{BMSSS18, HZ15b}.
We refer the reader to the work of the first author for further details~\cite{Lan23}.

\section*{Acknowledgements}
	
The first author was supported by the Ramón y Cajal programme (RYC2022-038372-I) and by grant PID2023-147202NB-I00 funded
by MICIU/AEI/10.13039/501100011033.
The second author was supported by ANID-FONDECYT Regular Nº1251121 grant.

\bibliographystyle{amsplain}
\bibliography{../bibliography.bib}

\providecommand{\bysame}{\leavevmode\hbox to3em{\hrulefill}\thinspace}
\providecommand{\MR}{\relax\ifhmode\unskip\space\fi MR }
\providecommand{\MRhref}[2]{%
  \href{http://www.ams.org/mathscinet-getitem?mr=#1}{#2}
}
\providecommand{\href}[2]{#2}
\begin{thebibliography}{10}

\bibitem{ABCM17}
P.~Allen, J.~B{\"o}ttcher, O.~Cooley, and R.~Mycroft, \emph{Tight cycles and
  regular slices in dense hypergraphs}, J. Combin. Theory Ser. A \textbf{149}
  (2017), 30--100.

\bibitem{AlonDefantKravitz2022}
N.~Alon, C.~Defant, and N.~Kravitz, \emph{The runsort permuton}, Adv. in Appl.
  Math. \textbf{139} (2022), Paper No. 102361, 18.

\bibitem{AKL+23}
J.~D. Alvarado, Y.~Kohayakawa, R.~Lang, G.~O. Mota, and H.~Stagni,
  \emph{Resilience for loose {H}amilton cycles}, arXiv:2309.14197 (2023).

\bibitem{Bar75}
Zs. Baranyai, \emph{On the factorization of the complete uniform hypergraph},
  Infinite and finite sets ({C}olloq., {K}eszthely, 1973; dedicated to {P}.
  {E}rd\H os on his 60th birthday), {V}ols. {I}, {II}, {III}, Colloq. Math.
  Soc. J\'anos Bolyai, vol. Vol. 10, North-Holland, Amsterdam-London, 1975,
  pp.~91--108. \MR{416986}

\bibitem{BMSSS17}
J.~{de O.} Bastos, G.~O. Mota, M.~Schacht, J.~Schnitzer, and F.~Schulenburg,
  \emph{Loose {H}amiltonian cycles forced by large $(k-2)$-degree---approximate
  version}, SIAM J. Discrete Math. \textbf{31} (2017), no.~4, 2328--2347.

\bibitem{BMSSS18}
\bysame, \emph{Loose {H}amiltonian cycles forced by large
  $(k-2)$-degree---sharp version}, Contributions to Discrete Mathematics
  \textbf{13} (2018), no.~2, 88--100.

\bibitem{BHS13}
E.~Bu{\ss}, H.~H{\`a}n, and M.~Schacht, \emph{Minimum vertex degree conditions
  for loose {H}amilton cycles in 3-uniform hypergraphs}, J. Combin. Theory Ser.
  B \textbf{103} (2013), no.~6, 658--678.

\bibitem{DH81}
D.~E. Daykin and R.~H{\"a}ggkvist, \emph{Degrees giving independent edges in a
  hypergraph}, Bull. Aust. Math. Soc. \textbf{23} (1981), no.~1, 103--109.

\bibitem{Dir52}
G.~A. Dirac, \emph{Some theorems on abstract graphs}, Proc. Lond. Math. Soc.
  \textbf{3} (1952), no.~1, 69--81.

\bibitem{DF13}
A.~Dudek and A.~Frieze, \emph{Tight {H}amilton cycles in random uniform
  hypergraphs}, Random Struct. Algorithms \textbf{42} (2013), no.~3, 374--385.

\bibitem{GHS+21}
L.~Gan, J.~Han, L.~Sun, and G.~Wang, \emph{Large {$Y_{k,b}$}-tilings and
  {H}amilton {$\ell$}-cycles in {$k$}-uniform hypergraphs}, J. Graph Theory
  \textbf{104} (2023), no.~3, 516--556.

\bibitem{GM18}
F.~Garbe and R.~Mycroft, \emph{Hamilton cycles in hypergraphs below the {D}irac
  threshold}, J. Combin. Theory Ser. B \textbf{133} (2018), 153--210.

\bibitem{GGJKO2020}
S.~Glock, S.~Gould, F.~Joos, D.~K{\"u}hn, and D.~Osthus, \emph{Counting
  {H}amilton cycles in {D}irac hypergraphs}, Comb. Probab. Comput. \textbf{30}
  (2020), no.~4, 631--653.

\bibitem{GH12}
C.~Grosu and J.~Hladk{\`y}, \emph{The extremal function for partial bipartite
  tilings}, Eur. J. Comb. \textbf{33} (2012), no.~5, 807--815.

\bibitem{GHM+23}
P.~Gupta, F.~Hamann, A.~Müyesser, O.~Parczyk, and A.~Sgueglia, \emph{A general
  approach to transversal versions of {D}irac-type theorems}, Bull. Lond. Math.
  Soc. \textbf{55} (2023), no.~6, 2817--2839.

\bibitem{HHM20}
H.~H{\`a}n, J.~Han, and P.~Morris, \emph{Factors and loose {H}amilton cycles in
  sparse pseudo-random hypergraphs}, Random Structures Algorithms \textbf{61}
  (2022), no.~1, 101--125.

\bibitem{HHZ20}
H.~Han, J.~Han, and Y.~Zhao, \emph{Minimum degree thresholds for {H}amilton
  {$(k/2)$}-cycles in {$k$}-uniform hypergraphs}, J. Combin. Theory Ser. B
  \textbf{153} (2022), 105--148.

\bibitem{HPS09}
H.~H{\`a}n, Y.~Person, and M~Schacht, \emph{On perfect matchings in uniform
  hypergraphs with large minimum vertex degree}, SIAM J. Discrete Math.
  \textbf{23} (2009), no.~2, 732--748.

\bibitem{HSW25}
J.~Han, L.~Sun, and G.~Wang, \emph{Minimum degree conditions for {H}amilton
  $\ell$-cycles in $k$-uniform hypergraphs}, Electron. J. Combin. \textbf{32}
  (2025), no.~1.

\bibitem{HZ15b}
J.~Han and Y.~Zhao, \emph{Minimum vertex degree threshold for loose {H}amilton
  cycles in $3$-uniform hypergraphs}, Journal of Combinatorial Theory, Series B
  \textbf{114} (2015), 70--96.

\bibitem{HZ16}
\bysame, \emph{Forbidding {H}amilton cycles in uniform hypergraphs}, J. Combin.
  Theory Ser. A \textbf{143} (2016), 107--115.

\bibitem{JLS23}
F.~Joos, R.~Lang, and N.~Sanhueza-Matamala, \emph{Robust {H}amiltonicity},
  arXiv:2312.15262 (2023).

\bibitem{KMP25}
A.~Kathapurkar, P.~Morris, and G.~Perarnau, \emph{A rainbow {D}irac theorem for
  loose {H}amilton cycles in hypergraphs}, arXiv:2501.07644 (2025).

\bibitem{KM15}
P.~Keevash and R.~Mycroft, \emph{A geometric theory for hypergraph matching},
  vol. 233, Mem. Amer. Math. Soc., 2015.

\bibitem{KMP23}
T.~Kelly, A.~M{\"u}yesser, and A.~Pokrovskiy, \emph{Optimal spread for spanning
  subgraphs of {D}irac hypergraphs}, J. Combin. Theory Ser. B \textbf{169}
  (2024), 507--541.

\bibitem{KSS98}
J.~Koml\'{o}s, G.~N. S\'{a}rk\"{o}zy, and E.~Szemer{\'e}di, \emph{On the
  {P}\'{o}sa-{S}eymour {c}onjecture}, J. Graph Theory \textbf{29} (1998),
  no.~3, 167--176.

\bibitem{KMO10}
D.~K{\"u}hn, R.~Mycroft, and D.~Osthus, \emph{Hamilton $\ell$-cycles in uniform
  hypergraphs}, J. Combin. Theory Ser. A \textbf{117} (2010), no.~7, 910--927.

\bibitem{KO06}
D.~K{\"u}hn and D.~Osthus, \emph{Loose {H}amilton cycles in 3-uniform
  hypergraphs of high minimum degree}, J. Combin. Theory Ser. B \textbf{96}
  (2006), no.~6, 767--821.

\bibitem{Lan23}
R.~Lang, \emph{Tiling dense hypergraphs}, arXiv:2308.12281 (2023).

\bibitem{LS22}
R.~Lang and N.~Sanhueza-Matamala, \emph{Minimum degree conditions for tight
  {H}amilton cycles}, J. Lond. Math. Soc. \textbf{105} (2022), no.~4,
  2249--2323.

\bibitem{LS24b}
\bysame, \emph{Blowing up {D}irac's theorem}, arXiv:2412.19912 (2024).

\bibitem{LS24a}
\bysame, \emph{A hypergraph bandwidth theorem}, arXiv:2412.14891 (2024).

\bibitem{LSV24}
R.~Lang, M.~Schacht, and J.~Volec, \emph{Tight {T}amiltonicity from dense links
  of triples}, arXiv:2403.14518 (2023).

\bibitem{LMM16}
J.~Lenz, D.~Mubayi, and R.~Mycroft, \emph{Hamilton cycles in quasirandom
  hypergraphs}, Random Struct. Algorithms \textbf{49} (2016), no.~2, 363--378.

\bibitem{LR25}
S.~Letzter and A.~Ranganathan, \emph{Exact supported co-degree bounds for
  {H}amilton cycles}, arXiv:2512.07751 (2025).

\bibitem{MP2023}
R.~Montgomery and M.~Pavez-Signé, \emph{Counting spanning subgraphs in dense
  hypergraphs}, Combin. Probab. Comput. \textbf{33} (2024), no.~6, 729--741.

\bibitem{Myc16}
R.~Mycroft, \emph{Packing $k$-partite $k$-uniform hypergraphs}, J. Combin.
  Theory Ser. A \textbf{138} (2016), 60--132.

\bibitem{MZ25}
R.~Mycroft and C~Z{\'a}rate-Guer{\'e}n, \emph{Positive codegree thresholds for
  {H}amilton cycles in hypergraphs}, arXiv:2505.11400 (2025).

\bibitem{NS20}
B.~Narayanan and M.~Schacht, \emph{Sharp thresholds for nonlinear {H}amiltonian
  cycles in hypergraphs}, Random Struct. Algorithms \textbf{57} (2020), no.~1,
  244--255.

\bibitem{PT22}
K.~Petrova and M.~Truji{\'c}, \emph{Transference for loose {H}amilton cycles in
  random 3-uniform hypergraphs}, Random Structures Algorithms \textbf{65}
  (2024), no.~2, 313--341.

\bibitem{RRS06}
V.~R{\"o}dl, A.~Ruci{\'n}ski, and E.~Szemer{\'e}di, \emph{A {D}irac-type
  theorem for $3$-uniform hypergraphs}, Comb. Probab. Comput. \textbf{15}
  (2006), no.~1-2, 229--251.

\bibitem{RRS08a}
\bysame, \emph{An approximate {D}irac-type theorem for $k$-uniform
  hypergraphs}, Combinatorica \textbf{28} (2008), no.~2, 229--260.

\bibitem{Sze76}
E.~Szemer{\'e}di, \emph{Regular partitions of graphs}, Colloq. Internat. CNRS
  \textbf{260} (1976), 399--401.

\end{thebibliography}

\appendix

\section{Auxiliary results}\label{sec:auxilliary-results-proofs}

In the following, we prove two auxiliary results concerning boosting and transitions, whose proofs are similar to their tight counterparts~\cite{LS24a}.

\subsection{Boosting} \label{sec:boosting}

In this section, we show \cref{lem:booster}.
We require the following simple fact.

\begin{observation}[{\cite[Observation 11.1]{LS24a}}]\label{obs:fractional-threshold}
	Let $1/r,\,1/s,\,\eps \gg 1/n$.
	Let $P$ be an $n$-vertex $s$-graph with $\delta_r(P) \geq (1-1/s + \eps) \binom{n-r}{s-r}$.
	Then~$P$ has a perfect fractional matching.
\end{observation}

\begin{proof}
	By monotonicity of the degree-types (\cref{fact:monotone-degrees}), we have $\delta_1(P) \geq (1-1/s + \eps) \binom{n-1}{s-1}$.
	Write $n = m k + q$ with $0 \leq q \leq k-1$.
	By the choice of $n$ and \cref{fact:matchingthresholds}, $P$ has a perfect matching after deleting any set of (exactly) $q$ vertices.
	We may thus obtain a perfect fractional matching of $P$ by averaging over all possible choices of $q$ vertices.
\end{proof}

\begin{proof}[Proof of \cref{lem:booster}]
	We abbreviate $\P= \dcon_\ell \cap \dspa_\ell$.
	Let $G$ be a $k$-graph on $n$ vertices that $(\delta_1,r_1,s_1)$-robustly satisfies $\P$.
	Set $P = \PG{H}{\P}{s_1}$, and 
	let $\delta = 1 - 1/s$. Then $\delta_1 = \delta + \eps$, and therefore $\delta_{r_1}(P) \geq (\delta + \eps) \binom{n-r_1}{s_1-r_1}$.
	Now set $Q = \PG{P}{\DegF{r_1}{\delta+\eps/2}}{s_2}$.
	In other words, $Q$ is the $s_2$-graph on $V(P)=V(G)$ with an $s_2$-edge $S$ whenever the $s_1$-graph $P[S]$ has minimum ${r_1}$-degree at least $\left(\delta+\eps/2\right) \binom{s_2-r_1}{{s_1-r_1}}$.
	By \cref{lem:inheritance-minimum-degree}, it follows that $\delta_{r_2}(Q) \geq  \delta_2 \tbinom{n-r_2}{s_2-r_2}$.
	
	Fix an arbitrary $s_2$-edge $S' \in E(Q)$.
	To finish, it suffices to show that $G[S']$ satisfies $\Del_q(\P)$.
	To this end, let $D \subset S'$ be a set of at most $q$ vertices, and let $S = S' \sm D$.
	Our goal is now to show that $G[S]$ satisfies $\P = \dcon_\ell \cap \dspa_\ell$.
	
	For the space property, note that $P[S]$ has a perfect fractional matching by \cref{obs:fractional-threshold}.
	Furthermore, note that $G[R]$ has a perfect fractional $\ell$-cycle tiling for every $s_1$-edge $R$ in $P[S]$ by definition of $P$.
	We may therefore linearly combine these matchings to a perfect fractional  $\ell$-cycle tiling of $G[S]$.
	
	For the connectivity property, consider disjoint $e,f \in G[S]^{(\ell)}$.
	We have to show that $e$ and~$f$ are in the same $\ell$-component of $G[S]$.
	Note that $e$ and $f$ together span ${r_1}=2\ell$ vertices, and $\delta_{r_1}(P[S]) \geq (\delta+\eps/2) \binom{s_2-r_1}{s_1-r_1}$.
	Hence there is an $s_1$-edge $R \in P[S]$ that contains both $e$ and $f$.
	Moreover,  $G[R] \in \dcon_\ell$, by definition of $\P$.
	Then it follows that $e$ and $f$ are in the same $\ell$-component of $G[R]$, and hence in the same component of $G[S]$.
\end{proof}

\subsection{Transitions}\label{sec:transition}

In this section, we show \cref{lem:transition}.
Our proof uses the following two technical results.
The first one is an Inheritance Lemma, similar to \cref{lem:inheritance-minimum-degree}.

\begin{lemma}[{\cite[Lemma 8.1]{LS24a}}]\label{lem:inheritance-minimum-degree2}
	For $1/k,\,1/r,\,\eps \gg 1/s \gg 1/n$ and $\delta \geq 0$, let $G$ be an $n$-vertex $k$-graph, and let $D$ be a subset of $d$-sets in $V(G)$.
	Suppose that for each $e \in D$ there exists $G_e \subseteq G$ such that $\deg_{G_e}(e) \geq (\delta + \eps) \tbinom{n-d}{k-d}$.
	Let $P$ be the $s$-graph consisting of all $s$-sets $S$ such that, for all $e \in D$ contained in~$S$, $\deg_{G_e[S]}(e) \geq (\delta + \eps/2) \tbinom{s-d}{s-d}$.
	Then $\delta_{r}(P) \geq  (1-e^{-\sqrt{s}}    )  \tbinom{n-r}{s-r}$.\qed
\end{lemma}

The following result tells us that a dense graph is guaranteed to contain a dense $(k-1)$-component.

\begin{lemma}[{\cite[Lemma 11.15]{LS24a}}]\label{lem:minimum-degree-component}
	Let $1/k,\, \eps \gg 1/n$  and $\delta \geq 2^{-1/(k-d)}$ for $1\leq d \leq k-1$.
	Let $G$ be a $k$-graph with $\delta_d(G)\geq (\delta +\eps) \binom{n-d}{k-d}$.
	Then $G$ contains a $(k-1)$-component $C \subset G$ with $\delta_d(C)\geq \delta^{k-d} \binom{n-d}{k-d}$.\qed
\end{lemma}

\begin{proof}[Proof of \cref{lem:transition}]
	Introduce $\eps,s_3$ with  $1/k,\,1/s_1 \gg \eps \gg s_2 \gg s_3 \gg 1/n$.
	Let $Q = \PG{G}{\P}{s_1}$, and note that by assumption $\delta_{2k} (Q) \geq (1-1/s_1^2) \binom{n-2k}{s_1-2k}$.
	Denote by $F$ the Hamilton $\ell$-framework of $\P$ that we are guaranteed by the assumption.
	For each $T \in Q$, let $C_T$ be the $[k]$-graph with $C_T^{(k)} = F(G[T])$ and no further edges (for the moment).
	Let $C$ be the union of the $[k]$-graphs $C_T$ over all edges $T \in Q$.
	Note that for each $e \in \partial_\ell(C)$ there is at least one edge $T \in Q$ such that $e \in \partial_\ell(C_T)$.
	For our purposes, we require a bit more.
	To this end, let us say that $e \in \partial_\ell(C)$ is \emph{well-connected} if there are at least $\eps \binom{n-\ell}{s_1-\ell}$ edges $T \in Q$ such that $e \in \partial(C_T)$.
	Crude double-counting reveals that there are at least $\eps \binom{n}{\ell}$ well-connected sets in $\partial(C)$.
	Now we define a $k$-bounded hypergraph (which we still call $C$) by adding to $C$ all well-connected $\ell$-tuples.
	We claim that $C$ satisfies $s_2$-robustly $\con_\ell \cap \spa_\ell$.
	
	To see this, define an $s_2$-graph $P$ on $V(G)$ by adding an edge $S$ whenever
	\begin{enumerate}[\upshape (i)]
		\item \label{itm:Q[S]-large-degree} $Q[S]$ has minimum $\ell$-degree at least $(1-e^{-\sqrt{s_2}})\binom{n-\ell}{s_2-\ell}$,
		\item \label{itm:Q[S]-well-connected-existence} $S$ contains a well-connected $\ell$-set and
		\item \label{itm:Q[S]-well-connected-degree} every well-connected $\ell$-set $e \subset S$ is contained in $\partial_\ell(C_T)$ for at least $(\eps/2) \binom{s_2-\ell}{s_1-\ell}$ edges $T \in Q[S]$.
	\end{enumerate}
	It follows by \cref{lem:inheritance-minimum-degree,lem:inheritance-minimum-degree2} that $P$ has minimum $2k$-degree at least $(1-1/s_2^2) \binom{n-2k}{s_2-2k}$.
	Consider an edge $S \in P$.
	In the following, we show that~$G[S]$ satisfies $\con_\ell \cap \spa_\ell$. 
	
	We begin with connectivity.
	Recall that the {$\ell$-adherence} $\adh_\ell(C[S]) \subset (C[S])^{(k)}$ is obtained by taking the union of the \tight $\ell$-components $\tc(e)$ over all $e \in (C[S])^{(\ell)}$.
	We have to show that $\adh(C[S])$ is a single vertex-spanning \tight $\ell$-component.
	To this end, we first identify an \tight $\ell$-component $J \subset C[S]$ and then show that $J$ contains the components $\tc(e)$.
	By choice, each $C_T$ is itself \tightly $\ell$-connected.
	{Moreover, for all $T,T' \in Q$ with $|T \cap T'| = s_1-1$, the edges of $C_{T}^{(k)} = F(G[T])$ and $C_{T'}^{(k)} = F(G[T'])$ are in the same \tight $\ell$-component of $G[S]$ thanks to the consistency condition of \cref{def:hamilton-framework}.}
	{This inspires us to select an $s_2$-vertex \tight $(s_1-1)$-component $Q'_S \subset Q[S]$ of minimum $\ell$-degree at least 
		\begin{align*}
			(1-e^{-\sqrt{s_2}})^k \binom{s_2-\ell}{s_1-\ell} \geq (1-e^{-\sqrt{s_2}/2}) \binom{s_2-\ell}{s_1-\ell}\,,
		\end{align*}
		which is possible by property~\ref{itm:Q[S]-large-degree} and \cref{lem:minimum-degree-component}.}
	Let $J \subset C[S]$ be formed by the union of the $k$-graphs $C_{T}$ with $T \in Q'_S$.
	As observed above, $J$ is \tightly $\ell$-connected.
	
	Now consider $e \in (C[S])^{(\ell)}$, which exists by property~\ref{itm:Q[S]-well-connected-existence}.
	To prove that $C$ satisfies $\con_\ell$, we show that $\tc_{C[S]}(e) = J$.
	By the above it suffices to show that $e$ is contained in an edge $T \in Q'_S$.
	Recall that by definition $e$ is well-connected in $G$, and hence $e$ is contained in $\partial_\ell(C_T)$ for at least $(\eps/2) \binom{s_2-\ell}{s_1-\ell}$ edges $T \in Q[S]$ by property~\ref{itm:Q[S]-well-connected-degree}.
	Since $\eps/2 > e^{-\sqrt{s_2}/2}$, this means that one of these edges is in $Q'_S$, and we are done.
	It follows that $C$ satisfies $\con_\ell$.
	
	For the space property, we proceed as in the proof of \cref{lem:booster}.
	First observe that $Q'_S$ has a perfect fractional $\ell$-cycle tiling.
	Indeed, by \cref{fact:matchingthresholds,obs:fractional-threshold}, the $s_2$-vertex $s_1$-graph $Q'_S$ still has a perfect fractional matching.
	Furthermore, note that $C[T]$ has a perfect fractional $\ell$-cycle tiling for every $s_1$-edge $T$ in $Q'_S$.
	We may therefore linearly combine these tilings to a fractional perfect $\ell$-cycle tiling of $C[S]$. 
\end{proof}

\section{Hamilton connectedness}\label{sec:connectedness}

In this section, we show \cref{thm:framework-connectedness-robust}.
Given \cref{prp:allocation-Hamilton-cycle}, the proof follows mutatis mutandis the one of its tight analogue~\cite[Theorem 5.8]{LS24a}.
We require the following two auxiliary results.
The statements have been simplified to the setting of constant size cluster and undirected hypergraphs.

\begin{lemma}[{\cite[Lemma 7.2]{LS24a}}]\label{lem:connecting-blow-ups}
	Let $\eps,  1/s_1 \gg 1/s_2 \gg c \gg 1/n$.
	{Suppose $m_1 \leq n/s_2$ and $m_2= (\log m_1)^c$ are positive integers.}
	Let  $G$ be an $n$-vertex $[k]$-graph.
	Let $\P$ be a family of $[k]$-graphs.
	Suppose that $G$ satisfies  $(\eps,s_1,s_2)$-robustly $\P$.
	Let $\cV$ be an $m_1$-balanced set family in $V(G)$ with $s_1$ clusters.
	Then $G$ contains an $s_2$-sized quasi $m_2$-balanced $\P$-blow-up $T(\cW)$ such that $\cW$ hits $\cV$, and such that the exceptional vertex of~$\cW$ is not contained in $\cV$.
\end{lemma}

\begin{lemma}[{\cite[Lemma 7.3]{LS24a}}]\label{lem:blow-up-support}
	Let $1/k,\,\eps,\, 1/s,\,1/m \gg \beta \gg 1/n$ and $1 \leq d \leq k$.
	Let~$G$ be an $n$-vertex $[k]$-digraph.
	Let~$\P$ be a family of $s$-vertex $[k]$-digraphs with $R^{(d)}$ non-empty for each $R \in \P$.
	Suppose that $\PG{G}{  \P   }{s}$ has at least $\eps n^s$ edges.
	Then there are at least $\beta n^{d}$ edges $e \in G^{(d)}$ such that there are $\beta n^{ms-d}$ many $m$-balanced $\P$-blow-ups in~$G$ that contain $e$ as an edge.
\end{lemma}

\begin{proof}[Proof of \cref{thm:framework-connectedness-robust}]
	Given $k,{\ell},s_1$, let $r = 2k$ and introduce $p_1$, $p_2$, $q_1$, $q_2$, $j_2$, $j_1$ and $\beta$ with
	\begin{align*}
		\frac{1}{\ell},\,\frac{1}{k},\,\frac{1}{r}, \frac{1}{s_1} \gg \frac{1}{p_1} \gg \frac{1}{p_2} \gg \frac{1}{q_1} \gg \frac{1}{q_2} \gg \frac{1}{j_2} \gg \frac{1}{j_1} \gg \frac{1}{s_2} \geq \frac{1}{s_3} \gg \frac{1}{n} \quad \text{and}  \quad 	\frac{1}{k}, \, \frac{1}{j_1}, \,\frac{1}{p_2} \gg \beta \gg \frac{1}{s_2} \, .
	\end{align*}
	For brevity, let $\mathsf{Q} = \dcon_\ell \cap \dspa_\ell$.
	Since $\P$ is a family of $s_1$-vertex $k$-graphs which admits a Hamilton $\ell$-framework and $G$ satisfies $s_1$-robustly $\P$;
	we can apply \cref{lem:transition} to find  an $n$-vertex $[k]$-graph $G' \subset G \cup   \partial_\ell G $ that $p_1$-robustly satisfies $\mathsf{Q}$.
	By \cref{corollary:booster2}, $G'$ satisfies $(1 - \exp(-\sqrt{p_2}), {r}, p_2)$-robustly $\Del_k(\mathsf{Q})$.

	We apply \cref{lem:blow-up-support} (with $\Del_k(\mathsf{Q})$ playing the rôle of $\P$) to reveal that there are at least $\beta n^{\ell}$ edges $e \in G'^{(\ell)}$ such that there are $\beta n^{j_1p_2-\ell}$ many $j_1$-balanced $\Del_k(\mathsf{Q})$-blow-ups in $G'$ that contain~$e$ as an edge.
	Let $C$ be obtained from $G'$ by deleting all $\ell$-edges that do not have this property.
	Now fix any $s_2 \leq {s} \leq s_3$ with $s \equiv k \bmod k - \ell$.
	We claim that $C$ satisfies ${s}$-robustly $\hamcon$, which suffices to conclude.
	
	Let $P = \PG{G'}{\mathsf{Q}}{p_2}$.
	Set $Q = \PG{P}{\DegF{{r}}{1-1/p^2_2}}{{s}}$.
	In other words, $Q$ is the ${s}$-graph with vertex set $V(G)$ and an ${s}$-edge $S$ whenever the induced $p_2$-graph $P[S]$ has minimum ${r}$-degree at least $\left(1 - 1/p^2_2\right) \binom{{s}-{r}}{p_2-{r}}$.
	Since $G'$ satisfies $\Del_k(\mathsf{Q})$ $(1 - \exp(-\sqrt{p_2}), {r}, p_2)$-robustly, we have in particular that $\delta_{{r}}(P) \geq (1 - \exp(-\sqrt{p_2})) \tbinom{n - {r}}{p_2 - {r}}$.
	Hence, by \cref{lem:inheritance-minimum-degree}, it follows that $\delta_{{r}}(Q) \geq (1-1/(2{s}^2)) \tbinom{n-{r}}{{s}-{r}}$.
	
	Given an ${s}$-set $S \subset V(Q)$ and $e \in C^{\ell}$, we say that $S$ is \emph{$e$-tracking} if there is a $j_1$-balanced $\Del_k(\mathsf{Q})$-blow-up in $C[S] \cup \{e\}$ that contains $e$ as an edge.
	We say that $S$ is \emph{tracking} if it is $e$-tracking for all $e \in C^{\ell}[S]$.
	Finally, we call an ${s}$-set $S \subset Q$ \emph{bueno} if $C^{\ell}[S]$ has at least two disjoint edges.
	Denote by $Q' \subset Q$ the $n$-vertex subgraph of tracking and bueno edges.
	
	\begin{claim}
		We have $\delta_{{r}}(Q') \geq  (1-1/{s}^2)   \tbinom{n-{r}}{{s}-{r}}$.
	\end{claim}
	
	\begin{proofclaim}
		Fix a set of ${r}$ vertices $D$.
		Since $\beta \gg 1/{s}$, by \cref{lem:inheritance-minimum-degree} (applied with $r=0$, $d=0$ and $\delta=0$) we deduce that there are at most $1/(4s^2)\binom{n - {r}}{{s} - {r}}$ sets $S$ containing $D$ for which the induced $[k]$-graph $C[S]$ has fewer than $(\beta/2) {s}^{\ell}$ $(\ell)$-edges.
		In the opposite case, $C[S]$ has at least $(\beta/2) {s}^{\ell} > \binom{{s}-1}{\ell-1}$ edges of size $\ell$, and hence must have at least two disjoint edges.
		In other words, the ${s}$-graph of bueno edges has minimum ${r}$-degree at least $(1 - 1/(4s^2))\binom{n - {r}}{{s} - {r}}$.
		
		Now we wish to apply \cref{lem:inheritance-minimum-degree2} in an auxiliary hypergraph.
		Let $H$ be the $(j_1 \cdot p_2)$-uniform hypergraph with an edge for every $j_1$-balanced $\Del_k(\mathsf{Q})$-blow-up with $p_2$-uniform edges in $G'$.
		By construction, each $e \in D$ belongs to at least $\beta n^{j_1p_2-(\ell)} \geq 2 \beta \binom{n - (\ell)}{j_1 p_2 - (\ell)}$ edges of $H$.
		We apply \cref{lem:inheritance-minimum-degree2} in $H$, with $j_1p_2$, ${s}$, $\beta$, $C$ playing the rôles of $k$, $s$, $\varepsilon$, $D$.
		We obtain that the ${s}$-uniform graph of tracking edges has minimum ${r}$-degree at least $(1 - 1/(4 s^2)) \binom{n - {r}}{{s} - {r}}$.
		
		Together with the above and the fact that $\delta_{{r}}(Q) \geq (1 - 1/(2s^2))\binom{n - {r}}{{s} - r}$, this confirms the claim.
	\end{proofclaim}
	
	Consider an edge $S \in Q'$, and let $D=C[S]$.
	Let $g,f \in D^{(\ell)}$ be disjoint, which exist since $S$ is bueno.
	To finish, we show that $D$ contains a Hamilton $(g,f,\ell)$-path.
	
	Since $Q' \subseteq Q$, from the definition of $Q$ we get that $\mathsf{Q}$ is $p_2$-robustly satisfied by~$D$.
	Using $1/k, 1/p_2 \gg 1/q_1 \gg 1/q_2 \gg 1/{s}$, we can apply \cref{lem:booster} twice and obtain that $D$ satisfies $\Del_{k}(\mathsf{Q})$ both $(1,q_1)$-robustly and $(2q_1,q_2)$-robustly.
	Introduce new constants $c, \eta$ such that $1/q_2 \gg c, \eta \gg 1/j_2$ hold, and let $m_1 = (\log {s})^{c}$ and $m_2 = (\log m_1)^{c}$.
	By \cref{prp:blow-up-cover} with ${s}$ playing the rôle of $n$,~$D$ has a $(q_1,q_2)$-sized $(m_1,m_2,\eta)$-balanced $\Del_{k}(\mathsf{Q})$-cover whose shape $F$ is a path.
	
	By definition of $Q'$, we have that $S$ is $g$-tracking.
	Hence, there is a $p_2$-sized $j_1$-balanced $\P$-blow-up $R^g(\cV^g)$ that contains $g$ as an edge.
	Let $x$ be the first vertex of $F$ with corresponding blow-ups $R^{x}(\cV^{x})$.
	Select a set of size $j_1$ inside each set of the family $\cV^{x}$; let $\cV^{x}_2$ be those sets.
	This can be done keeping the sets disjoint from the vertices of $R^g(\cV^g)$.
	In this way, we obtain a $j_1$-balanced $q_1$-sized blow-up $R^{x}(\cV_2^{x})$, and so that
	$\cV_2^{x} \cup \cV^g$ is a $j_1$-balanced family with $p_2 + q_1 \leq 2q_1$ clusters.
	We apply \cref{lem:connecting-blow-ups} with $\cV_2^{x} \cup \cV^g$ in place of $\cV$.
	This gives a $j_2$-balanced $\P$-blow-up $R^{gx}(\cW^{gx})$ such that $\cW^{gx}$ hits $\cV^g$ with $\cW^{gx}_{g} \subset \cW^{gx}$ and $\cV^{x}$ with $\cW^{gx}_x \subset \cW^{gx}$.
	We repeat the same process with $f$ to obtain $\P$-blow-ups $R^f(\cV^f)$ and  $R^{fy}(\cW^{fy})$, where $y$ is the last vertex of the path $F$.
	
	Note that these four blow-ups might intersect each other (quite a lot) outside of the dedicated hitting areas.
	To fix this, we may select $j_1'$-balanced and $j_2'$-balanced, respectively, subfamilies that are vertex-disjoint (except for the hitting areas), keeping the names for convenience.
	This can be done with  $j_1' = j_1/3$ and $j_2' = j_2/3$ using a random partitioning argument.
	We delete the vertices of these additional four blow-ups from the cover  $(\{\cV^v\}_{v \in V(F)},  \{\cW^{e}\}_{e \in F})$, keeping again the names for convenience.
	Since the additional blow-ups have only $2q_2(j_1'+j_2')$ vertices, $(\{\cV^v\}_{v \in V(F)},  \{\cW^{e}\}_{e \in F}))$ is still $(m_1,m_2,2\eta)$-balanced.
	Let $F'$ be the path obtained from $F$ by extending its ends with two vertices $g$ and $f$.
	
	To finish, we identify the vertices of $F$ with $1,\dots,\ell$ following the order of the path.
	For each $i \in V(F)$, we select disjoint oriented edges $e_i,f_i \in R^{i}(\cV^{i})$ with $e_1 = f$ and $e_\ell = g$, where $e_1$ and $e_\ell$ are oriented as in the assumption and all other edges are arbitrarily oriented.
	We then apply \cref{prp:allocation-Hamilton-cycle} to find a \tight Hamilton $(f_i,e_{j},\ell)$-path $P^{ij}$ in $R^{ij}(\cW^{ij})$ for each $i \in [\ell]$ and $j=i+1$.
	Note that the families $\cV^{i}$ are still $(1\pm 2\eta)m_1$-balanced (resp. $(1\pm 2\eta)j_1'$-balanced) after deleting the vertices of these paths.
	We may therefore finish by applying \cref{prp:allocation-Hamilton-cycle} to find a \tight Hamilton $(e_i,f_{i},\ell)$-path $P^{i}$ in $R^{i}(\cW^{i})$ for each $i \in [\ell]$ with index computations taken modulo $\ell$.
\end{proof}

\end{document}